\documentclass[12pt, reqno]{amsart}

\usepackage{amsmath, amsthm, amscd, amsfonts, amssymb, graphicx, color, mathrsfs}

\usepackage{eucal}
{\setlength\arraycolsep{2pt}}
\usepackage{cite}

\usepackage[all]{xy}
\usepackage{slashed}

\usepackage[utf8]{inputenc}
\usepackage[colorlinks,urlcolor=blue,citecolor=blue,linkcolor=blue]{hyperref}
\usepackage{cleveref}

\newtheorem{theorem}{Theorem}
\newtheorem{lemma}{Lemma}

\newtheorem{remark}{Remark}
\newtheorem{proposition}{Proposition}
\newtheorem{definition}{Definition}

\newtheorem{corollary}{Corollary}

\begin{document}
\setcounter{page}{1}

\title[Higher order coercive inequalities ]{Higher order coercive inequalities}
	
	\author[Y. Wang]{Yifu Wang}
	\address{
		Yifu Wang:
		Department of Mathematics
		Imperial College London
		180 Queen's Gate, London SW7 2AZ
         United Kingdom
		{\it E-mail address} {\rm yifu.wang17@imperial.ac.uk}
		}

		\author[B. Zegarlinski]{Boguslaw Zegarlinski}
	\address{
		Boguslaw Zegarlinski:
		Department of Mathematics
		Imperial College London
		180 Queen's Gate, London SW7 2AZ
         United Kingdom
		{\it E-mail address} {\rm b.zegarlinski@imperial.ac.uk}
		}
		
	 \thanks{The first author was supported by the Imperial College Scholarship}

	\begin{abstract} 
		We study the higher order $q$- Poincar{\'e} and other coercive inequalities for a class probability measures satisfying Adam's regularity condition.
	\end{abstract} 

		\maketitle



	\tableofcontents
	
\label{Section 1}
 \section{ Introduction} \label{Intro}
Studies of coercive inequalities for probability measures, in particular including Poincar{\'e} and Log-Sobolev type inequalities and their applications, has a long and interesting history see e.g. \cite{ABC}, \cite{BGL}, \cite{GZ}, \cite{DGS},
..., \cite{Bo} and references therein. Despite enormous progress in this area there is still a lot of widely open and challenging problems in particular in relation to analysis on Lie groups and infinite dimensional analysis, for some starting pointers in this direction the reader could consult e.g. \cite{HZ}, \cite{Mar}, \cite{BB},... and others. In this paper we study in a systematic way the higher order inequalities in functional spaces associated to a probability measure. In case of Lebesgue measure an extensive literature exists see e.g. \cite{Zi} and references therein. For probability measures first work done in the direction
of higher order bounds is \cite{R} with more extensive development achieved in \cite{A}. In particular in the latter work \cite{A}, the author considered a class of probability measures for which the logarithm of density of the probability measure was given as a function of regularised distance in $\mathbb{R}^n$ and satisfied some pointwise regularity conditions. The coercive inequalities  obtained in both papers were defective in the sense that besides the terms containing higher order derivatives it contained also terms dependent on $L_p$ norms. Expanding on techniques developed in \cite{HZ}, in the present paper we
prove tight coercive inequalities.\\
To this end we study first higher order Poincar{\'e} inequalities for a class of probability measures containing those considered in \cite{A}. Although getting optimal constants is currently challenging for higher order inequalities, we provide a constructive method to get some estimates of them. \\
Later we use the higher order Poincar{\'e} inequalities to obtain tight coercive inequalities in the form of bounds from below on $L_p$ norms of higher order derivatives of a function by suitable Orlicz norms of the function with  removed part which is nullified by the higher order derivatives. Unlike in case of Log-Sobolev inequality we do not have direct relation with tight coercive inequalities involving suitable functionals (as e.g. in \cite{BoG}, \cite{BR}, \cite{BG} \cite{Ch}, \cite{BCR}, \cite{BoZ1}, \cite{BoZ2}, \cite{RZ}). 
Thus in our context, in order to get a generalisation of Holley-Stroock perturbation lemma \cite{HS} it is more natural to consider norm dependent minimizer. We show that in some sense they are equivalent to the classical ones, but they are of other potential interest too if one would like to study concentration phenomena in various functional spaces. \\
Later, by developing further our techniques, we obtain quantitative estimates on the constants in the Adam's bounds of \cite{A}. This allows us to show equivalence of $W_{k,p}(\mu)$ norms and norms related to the Markov generator associated to the natural Dirichlet form with the given probability measure.
We remark that an extensive literature exists in relation to first order bounds with use of the Riesz transform see e.g. \cite{AT}, \cite{X}, \cite{Bo} and references therein. To get estimates of higher order some more effort was necessary.\\
Finally at the end of the paper we discuss decay to equilibrium properties on various spaces.
The motivation for this study for us comes from the works \cite{KOZ}, \cite{INZ}, \cite{Z} and \cite{YSW}, where various applications of higher gradient estimates were exploited in the infinite dimensional setting. One may hope that in the future further development of methods based on higher order estimates may help in better understanding difficult problems of decay to equilibrium for Markov semigroups in infinite dimensional context. 

 \label{Section 2}
\section{  Adams Inequalities} 
\label{  Sec.Adams Inequalities} 
Let $\lambda$ denotes the Lebesgue measure on $\mathbb{R}^n$. With a twice differentiable real function  $U$ satisfying
\[ 0< \int e^{-U}d\lambda< \infty  \] 
we define 
\[d\mu \equiv e^{-U}d\lambda \]

We say that $U$ is locally bounded if it is bounded on any ball $B(0,r)$ with radius $r\in(0,\infty)$ centered at the origin.\\  
In \cite{A} the following condition was explored.\\

\textbf{Adams Regularity Conditions}: $\exists \varepsilon, C\in(0,\infty)$

\[ \sum_{|\alpha|=2}|\nabla^\alpha U|\leq C (1+|\nabla U|)^{2-\varepsilon}. \tag{ARC}
   \]\\
   
   For $\mathbf{\alpha}\equiv(\alpha_j\in\mathbb{Z}^+)_{j=1,..,n}$ we set
\[|\mathbf{\alpha}|\equiv\sum_{j=1,..,n}\alpha_j
\]
and denote $\nabla^\alpha\equiv \prod_{j=1,..,n}\partial_j^{\alpha_j}$.
For $m\in\mathbb{N}$ and $p\geq 1$, define
\[
\|f\|_{m,p}^p\equiv \|f\|_{m,p,\mu}^p\equiv \sum_{\alpha:|\alpha|\leq m} \mu\left(|\nabla^\alpha f|^p\right) 
\]
and let $W_{m,p}(\mu)$ denote the closure of compactly supported smooth functions 
$\mathcal{C}^\infty_0$ with respect to this norm.\\
Let $A$ be a nonnegative nondecreasing function on $\mathbb{R}^+$ satisfying $\Delta_2$ condition, i.e. with some $K_0\in\mathbb{R}^+$ $\forall t\in\mathbb{R}^+$
 \[ A(2t) \leq  K_0 A(t).
 \]
and define
\[ \Phi_{A,p}(t)\equiv |t|^p A(\log^\ast |t|),\]
with $\log^\ast t\equiv \max\{1,\log (t)\}$.
   
\begin{theorem} \label{AdamsThm.1}([Adams'JFA1979, Theorem 1]) \\
 If $U$  satisfies the (ARC) regularity condition   and if, for some $a\in(0,\infty)$, the following inequality holds
\[ A(U(x))\leq a(1+|\nabla U|)^{mp}, \qquad a.e.,
\]
then there exists a constant $K\in(0,\infty)$ such that for all $f\in W_{m,p}(\mu)$,
\[
\mu\Phi_{A,p}(f) \leq K \left( \|f\|_{m,p}^p + \Phi_{A,p}(\|f\|_p) \right).
\]

\end{theorem}

\begin{lemma} \label{AdamsLem.2} ([Adams'JFA1979, Lemma B]) \\
There exists a constant $\tilde K$ such that for every $f\in W_{m,p}(\mu)$,
\[ \int |f|^p \left(1+|\nabla U|\right)^{mp} d\mu \leq
\tilde{K} \|f\|_{m,p}^p
\]
in particular 
\[
\int |f|^p |\nabla U| d\mu \leq
\tilde{K} \|f\|_{1,p}^p
\]
\end{lemma} 
\begin{proof}
In order to get some possibly useful information about the constants we provide here selfcontained arguments following \cite{HZ}. 
First of all from Leibnitz rule for the gradient we have
\[
(\nabla f)e^{-U} = \nabla(fe^{-U}) +(\nabla U)fe^{-U}
\]
Multiplying both sides by the subunit vector $\nabla d$, with $d$ being a distance function which is regularised on small distances, integrating with the Lebesgue measure and performing integration by parts on the right hand side, one obtains
\[
\int \nabla d\cdot \nabla f d\mu = \int f\left(\nabla d\cdot \nabla U - \Delta d \right) d\mu  
\]
Assuming 
\[\nabla d\cdot \nabla U - \Delta d \geq |\nabla U| - D\]
for some $D \in(0,\infty)$ and using our assumption $|\nabla d|\leq 1$, we get 
\begin{equation} \label{Alem.e1}
  \int f\left(1+| \nabla U |\right) d\mu  \leq \int |\nabla f| d\mu
\end{equation}
Next replacing $f$ by $|f|^p \left(1+| \nabla U |\right)^{mp-1}$, we have
\begin{equation} \label{Alem.e1'}
  \int |f|^p \left(1+| \nabla U |\right)^{mp}d\mu  \leq \int |\nabla \left( |f|^p \left(1+| \nabla U |\right)^{mp-1}|\right) d\mu
\end{equation}
We discuss the right hand side as follows
\begin{equation} \label{Alem.e2}
\begin{split}
\int |\nabla \left(|f|^p \left(1+| \nabla U |\right)^{mp-1}\right)  d\mu
  \leq  p \int |\nabla  f|\; |f|^{p-1} \left(1+| \nabla U |\right)^{mp-1}| d\mu \\
  + (mp-1) \int |f|^p \left(1+| \nabla U |\right)^{mp-2} |\nabla^2 U|  d\mu
\end{split}
\end{equation} 
Applying Youngs inequality  to the first term on the right hand side, with dual indices $\frac1p+\frac1q=1$, we have
\[ \begin{split}
 p \int |\nabla  f|\; |f|^{p-1} \left(1+| \nabla U |\right)^{mp-1}  d\mu 
 \qquad \qquad \qquad \qquad \qquad\qquad \qquad \qquad\qquad \qquad
\\
 \leq \epsilon^{-p/q} q^{-p/q}\frac1p\int |\nabla  f|^p \left(1+| \nabla U |\right)^{(m-1)p} d\mu  
+ 
\varepsilon  \int   |f|^p \left(1+| \nabla U |\right)^{mp} d\mu
\end{split} 
\]
For the second term on the right hand side of (\ref{Alem.e2}), if we assume that 
\begin{equation} \label{Alem.e3}
  |\nabla^2 U|   \leq C \frac{\left(1+| \nabla U |\right)^2}{\eta(| \nabla U |)}
\end{equation}
with some function $1\leq \eta(x)\to_{x\to\infty} \infty$, then 
for any $\varepsilon\in(0,1)$ there is a constant $C_{\frac{\varepsilon}{(mp-1)}}\in(0,\infty)$ such that
\begin{equation} \label{Alem.e3'}
  |\nabla^2 U|   \leq \frac{\varepsilon}{(mp-1)}  \left(1+| \nabla U |\right)^2 + C_{\frac{\varepsilon}{(mp-1)}}
\end{equation}
and hence
\begin{equation} \label{Alem.e4}
 (mp-1) \int |f|^p \left(1+| \nabla U |\right)^{mp-2} |\nabla^2 U|  d\mu
  \leq   \varepsilon  \int |f|^p \left(1+| \nabla U |\right)^{mp}  d\mu
  + C_{\frac{\varepsilon}{(mp-1)}} \int |f|^p  d\mu 
\end{equation}
Combining (\ref{Alem.e1'})-(\ref{Alem.e4}), we arrive at
\begin{equation} \label{Alem.e5}
\begin{split}
  \int |f|^p \left(1+| \nabla U |\right)^{mp}d\mu  \leq  \\
  \epsilon^{-p/q} q^{-p/q}\frac1p\int |\nabla  f|^p \left(1+| \nabla U |\right)^{(m-1)p} d\mu  
+ 
\varepsilon  \int   |f|^p \left(1+| \nabla U |\right)^{mp} d\mu \\
+ \varepsilon  \int |f|^p \left(1+| \nabla U |\right)^{mp}  d\mu
  + (mp-1)C_{\frac{\varepsilon}{(mp-1)}} \int |f|^p  d\mu \\
  \end{split}
\end{equation}
Choosing $\varepsilon\in(0,\frac12)$ we arrive at the following bound
\begin{equation} \label{Alem.e6}
  \int |f|^p \left(1+| \nabla U |\right)^{mp}d\mu  \leq   
  \tilde C\int |\nabla  f|^p \left(1+| \nabla U |\right)^{(m-1)p} d\mu  
+ 
\tilde D \int |f|^p  d\mu  
\end{equation}
with some constant $\tilde C\equiv \tilde C(\varepsilon,p)\in (0,\infty)$ and $\tilde D\equiv \tilde D(\varepsilon,m,p)\in (0,\infty)$. From here we proceed  by induction
to conclude with the following bound
\begin{equation} \label{Alem.e7}
  \int |f|^p \left(1+| \nabla U |\right)^{mp}d\mu  \leq   
  C\sum_{k=1,..,m}\int |\nabla^k  f|^p   d\mu  
+ 
D \int |f|^p  d\mu  
\end{equation}

\end{proof} 

In an earlier paper \cite{R}, besides other interesting results, the following case was considered. 
Suppose $r > 0$ be such that
\[ |U|^r\leq a\left( \frac14 |\nabla U|^2 - \frac12 \Delta U + b \right) \]
with some $a,b\in(0,\infty)$. Then for $m=1$ and $p=2$ the inequality holds with  
$A(s)=|s|^r$.\\

In both cases, \cite{A} and \cite{R}, the results did not generalise to the infinite dimensional setting.

As an illustration of the results, in \cite{A} a major class of examples was provided in the form of $U=U(d)$ where $d(x)\equiv d(x,0)$ was a smooth function which coincided with the Euclidean distance function outside a ball centered at the origin. 

Since the function $A(t)$ satisfies doubling condition, one can show the following  bounded perturbation lemma, which includes non-regular $U$'s, without much of a problem.

\label{Thm2} %
\begin{theorem} \label{BddPertAdamsThm.1}   
 Suppose there exists a constant $K\in(0,\infty)$ such that for all $f\in W_{m,p}(\mu)$,
\[
\mu\left(\Phi_{A,p}(f) \right)\leq K \left( \|f\|_{m,p,\mu}^p + \Phi_{A,p}(\|f\|_{p,\mu}) \right).
\]
Suppose 
\[ d\tilde\mu \equiv e^{-V}d\mu\]
with a bounded measurable function $V$. Then there is a constant $\tilde{K}\in(0,\infty)$ such that for all $f\in W_{m,p}(\tilde\mu)$,
\[
\tilde\mu\left(\Phi_{A,p}(f)\right) \leq \tilde{K} \left( \|f\|_{m,p, \tilde\mu}^p + \Phi_{A,p}(\|f\|_{p,\tilde\mu}) \right).
\]
where on the right hand side the norms are with the measure $\tilde{\mu}$. 
\end{theorem}

We remark that one has the following equivalence of norms.

\label{Thm3} %
\begin{theorem} \label{Thm_W_kpNorms}
Under the conditions of Lemma \ref{AdamsLem.2}, suppose $U$ is locally bounded and $|\nabla U(x)|\to_{|x|\to\infty} \infty$, then there exists a constant $C\in(0,\infty)$ such that
\[
\|f\|_{m,p}^p \leq C\; \mu \left(|\nabla^m f|^p + f^p\right)
\]
for any $f\in W_{m,p}(\mu)$.
Thus the norms $\|f\|_{m,p}$ and
\begin{equation}
\| f \|_{k,p}^{\sim} \equiv \| f \|_p + \| \nabla^k f\|_p
\end{equation}
are equivalent.
\end{theorem}
\begin{proof}
Let  $B\equiv B(0,R)$ be a ball of radius $R$ centered at the origin.
We note a property that for Lebesgue measure $\lambda$ , for any $1\leq j< m$, one has \cite{A_bk}
\[  \int_{B} |\nabla^j  f|^p d\lambda \leq c_B \left( \int_{B} |\nabla^m f|^p  + |f|^pd\lambda \right).
\]
 with some constant $c_B$ independent of a function $f$.
Then using local boundedness of $U$, 
with $\chi_B$ denoting the characteristic function of $B$, we get
\begin{equation} \label{Eqn1}
 \begin{split}
 \mu \left( |\nabla^j f|^p \chi_B\right)\leq e^{-\inf_{B}U}\int_{B} |\nabla^j  f|^p d\lambda \\
 \leq e^{-\inf_{B}U} c_B \int_{B} \left(|\nabla^m f|^p  + |f|^p\right) d\lambda \\
  \leq e^{osc_{B}U} c_B \mu \left(|\nabla^m f|^p  + |f|^p\right)  
 \end{split}
\end{equation} 
where $osc_{B}U\equiv \sup_B(U)-\inf_B(U)$.
On the other hand using Lemma \ref{AdamsLem.2} and our assumption that $|\nabla U(x)|\to_{|x|\to\infty} \infty$, we have
\[\eta_R\equiv  1+\inf_{B^c}|\nabla U| \to_{R\to\infty} \infty\] 
and hence
\begin{equation} \label{Eqn2}
\begin{split}
 \mu \left( |\nabla^j f|^p \chi_{B^c}\right)\leq \frac1{\eta_R^{m-j}}\; \int_{B^c}  |\nabla^j f|^p \left(1+|\nabla U|\right)^{(m-j)p} d\mu \\
 \leq  \frac1{\eta_R^{m-j}}\; \mu \left(|\nabla^j f|^p \left(1+|\nabla U|\right)^{(m-j)p}\right) d\mu \\
 \leq \frac{K}{\eta_R^{m-j}}\;  \|f\|_{m,p}^p
 \end{split}
\end{equation} 
 Combining bounds (\ref{Eqn1})-(\ref{Eqn2}),  we get 
 \[  
  \mu \left( |\nabla^j f|^p  \right)\leq  e^{osc_{B}U} c_B \mu \left(|\nabla^m f|^p  + |f|^p\right)  +  \frac{K}{\eta_R^{m-j}}\;  \|f\|_{m,p}^p
 \] 
 Hence 
 \[  \begin{split}
  \sum_{1\leq j\leq m-1}\mu \left( |\nabla^j f|^p  \right)\leq \left[(m-1) e^{osc_{B}U} c_B +\frac{K}{\eta_R^{m-j}}\right]\mu \left(|\nabla^m f|^p  + |f|^p\right)  \\
  +     \sum_{1\leq j\leq m-1}\frac{K}{\eta_R^{m-j}}\;  \mu \left( |\nabla^j f|^p  \right) .
\end{split}
 \]
 Taking $R\in(0,\infty)$ sufficiently large so that
 \[
\frac{K}{\eta_R} <1, 
 \]
 this implies the desired property that there exists a constant $C\in(0,\infty)$ such that for all sufficiently large $R\in(0,\infty)$, we have
\[
\|f\|_{m,p}^p \leq C\; \mu \left(|\nabla^m f|^p + f^p\right)
\]
for any $f\in W_{m,p}(\mu)$. 
 
\end{proof}
\begin{remark}
We remark that other interesting applications of local perturbation technique can be found in the literature, see e.g. \cite{Ai}, \cite{DeS}, \cite{HS}, \cite{HZ},..., and others.
\end{remark}
\label{Section 3}
\section{  Poincare's Inequalities of Higher Order} 
\label{Poincare Ineqs of Higher Order}

We will show that under the conditions of Theorem \ref{Thm_W_kpNorms}, we also have a family of Poincar{\'e} inequalities satisfied. \\
Later on we denote by $\mathcal{P}_{k}$ the set of all polynomials on $\mathbb{R}^n$ of order $k$. 
\begin{lemma} \label{Lem.Minimizer}
Let $q\in (1, \infty] $, $u \in L_q(\mu)$. Then there exists a unique $w \in \mathcal{P}_{k-1}$ such that 
\begin{equation}
|| u - w ||_q = inf\{ ||u- \alpha||_q : \alpha \in \mathcal{P}_{k-1} \} \label{e1.22}
\end{equation}

\begin{proof}
Let $d(u, q)=  inf\{ ||u- \alpha||_q : \alpha \in \mathcal{P}_{k-1} \}$. 
Firstly, we claim that if such the minimizer $w$ exists, we must have $||w||_q \leq 2||u||_q$. Otherwise, 
\begin{equation}
||u-w||_q \geq ||w||_q - ||u||_q > 2 ||u||_q- ||u||_q = ||u||_q
\end{equation}
a contradiction to the definition of $w$. Hence we have
\begin{equation}
d(u, q)=  inf\{ ||u- \alpha||_q : ||\alpha||_q \leq 2 ||u||_q  \text{  and  }  \alpha \in \mathcal{P}_{k-1}\}
\end{equation}
Let $ \mathcal{P}_{k-1}^u = \{ \alpha \in \mathcal{P}_{k-1} | \|\alpha\|_q \leq 2 \|u\|_q \}$. 
Recall that in our setup $1<q<\infty$,  $L_q$ is reflexive. If $\{ \alpha_i \} \subset \mathcal{P}_{k-1}^u  $ is such that $||u- \alpha_i||_q \rightarrow d$, then there exists $w \in L_q$ so that $\alpha_i \rightharpoonup w$ in $L_q$. Since $\mathcal{P}_{k-1}^u$ is closed and convex, it is weakly closed. Thus $w \in \mathcal{P}_{k-1}^u$. Also notice that
\begin{equation}
\| u-w \| \leq \liminf \| u- \alpha_i \| = d(u,q)
\end{equation}
We obtain the existence. 

Next, if $u$ itself is a polynomial, then the uniqueness is trivial. Otherwise, suppose there are $w_1$ and $w_2$ satisfying \eqref{e1.22}, then so is $\lambda w_1 + (1- \lambda) w_2$. But
\begin{equation}
\begin{array}{lll}
d(u,q) &= & \| u- \lambda w_1 + (1- \lambda) w_2 \|   \\ [1.2 mm]
&= & \| \lambda (u-w_1) + (1- \lambda) (u-w_2) \|   \\[1.2 mm]
& \leq & \lambda \|u-w_1\| + (1- \lambda) \|u-w_2\| \\[1.2 mm]
& =& d(u,q)
\end{array}
\end{equation}
Since $L_q(\mu)$ is strictly convex, either $u=w_1$, or $u-w_1= \mu (u-w_2)$ for some $\mu$. Both of these cases indicate that $u$ itself is a polynomial, a contradiction. 

For $q= \infty$, if $u$ is continuous,  this is the Chebyshev Equioscillation Theorem.
 Otherwise, we can use density argument. 
\end{proof}
\end{lemma}

\begin{definition} \label{Def.1}
Given an integer $k$,  $q\in (1, \infty] $ and $u \in L_q(\mu)$, the polynomial in Lemma \ref{Lem.Minimizer} is called the $L_q$ minimizing polynomial of $u$ of order $k-1$, and is denoted as $M_{k,q}(u)\equiv M_{k,q,\mu}(u)$. 
\end{definition}
In the framework of $L_p$ spaces with Lebsegue measure in balls  the following inequalitites were established many years ago, see e.g. \cite{Z}. 
Let
\[
|\nabla^k f|^p\equiv \sum_{|\alpha|=k} |\nabla^\alpha f|^p.
\]

\textit{Let $B\equiv B(0,r)\subset\mathbb{R}^n$ be an open ball with radius $r$ centred at the origin. There exists a constant $\tilde{c}_{k,p}(B)\equiv \tilde{c}_{k,p}(B,\lambda)\in(0,\infty)$ such that}

\begin{equation} \label{IPk,p}
\tag{IP\textsubscript{k,p}($\lambda$,B)}
\int_B |f-M_{k,p,\lambda}(f)|^p\; d\lambda \leq \tilde{c}_{k,p}(B)  \int |\nabla^k f|^p\; d\lambda
\end{equation}

With the coercive inequalities we established previously, we can obtain the higher order Poincare's inequalities for other probability measures. 

\begin{theorem} \label{Thm.IP{k,q}}
Suppose $1<q<\infty$ and $k\in\mathbb{N}$. Let $d \mu= e^{-U} d \lambda$ be such that  $U$ satisfies Adams regularity condition, is locally bounded and $|\nabla U(x)|\to_{|x|\to\infty}\infty$.   Then there exists $ {c}_{k,q}\in(0,\infty)$ such that
\begin{equation}
\mu |f- M_{k,q, \mu}(f)|^q \leq {c}_{k,q} \mu |\nabla^k f|^q \label{e1.PI_kq}
\end{equation}
\end{theorem}

\begin{proof} We expand on the idea of \cite{HZ}.
First we note that with $w\equiv M_{k,q, \lambda,B}(f)$, defined with normalised Lebesgue measure in a ball $B_r\equiv B(0,r)$, we have
\begin{equation}
\begin{array}{ll}
\int  |f- M_{m,q, \mu}(f) |^q d\mu 
\leq &  \int  |f- w |^q d\mu \\ [1.2mm]
&\leq  \int_{ \{ B_r \} }  |f- w|^q d\mu + \int_{ \{ B_r^c \} }  |f- w |^q d\mu 
\end{array}
\end{equation}
For the first part, we have
\begin{equation} \label{e1.3.1}
\begin{array}{lll}
\int_{ \{  B_r  \} }  |f- w |^q d\mu  &\leq & e^{\inf_{B_r} (U)} \int_{B_R} |f- w |^q d\lambda \\[1.2mm]
&\leq & e^{\inf_{B_r} (U)}\tilde{c}_{k,p}(B)  \int_{B_R} |\nabla^k f|^q d\lambda \\[1.2mm]
&\leq & e^{osc_{B_r} (U)}\tilde{c}_{k,p}(B) \int |\nabla^k f|^q d\mu
\end{array}
\end{equation}
For the second part, we note that with
\[
L\equiv \inf_{B_r^c}(1+|\nabla U|),
\]
by Lemma \ref{AdamsLem.2} and Theorem \ref{Thm_W_kpNorms}, we have 
\begin{equation} \label{e1.3.2}
\begin{array}{ll}
 & \int_{ B_r^c }  |f- w |^q d\mu  \\[1.2mm]
\leq & \frac{1}{L} \int_{ B_r^c}  |f-w |^q \eta d\mu  \\[1.2mm]
\leq & \frac{1}{L} \int_{ B_r^c}  |f- w |^q (1+|\nabla U|)^{kq}\\
\leq & \frac{K}{L} \int |\nabla^k f|^q d\mu +\frac{K}{L} \int_{ B_r^c}  |f- w |^q d\mu
\end{array}
\end{equation}
Choosing $r\in(0,\infty)$ sufficiently large so that $L> K$, we obtain
\begin{equation}
 \int_{  B_r^c }  |f- w |^q d\mu \leq  \frac{K}{(L- K)} \int |\nabla^k f|^q d\mu
\end{equation}
Combining equations (\ref{e1.3.1})-(\ref{e1.3.2}), we get
\[ \int  |f- w |^q d\mu \leq C \int |\nabla^k f|^q d\mu \]
with 
\[C \equiv \max\left\{e^{osc_{B_r} (U)}\tilde{c}_{k,p}(B) ,\frac{K}{(L- K)}\right\}\]
\end{proof}

\newpage
\noindent\textbf{Downhill Induction $\mathbf{L}_q$-Case.} 
\label{Downhill Induction}

\noindent Here we provide constructive arguments to demonstrate that in fact in our setup $(PI_{1,q})$  implies the following inequalities of higher order.\\

\begin{theorem} \label{Thm.2.PI_1q}
For $q\in(1,\infty)$, there exists $C_{k,q}\in(0,\infty)$ such that

\[ \tag{$PI_{k,q}$}
\mu\left|f-m_{k,q}(f)\right|^q\leq C_{k,q} \sum_{|\bold{\alpha}|=k} 
\mu |\nabla^{\bold{\alpha}} f|^q,\]
with some polynomial $m_{k,q}(f)$ of order $k-1$.
\end{theorem}
\begin{proof}
Assume that  $(PI_{1,q})$ holds. Then
\[\begin{split}
\sum_{|\bold{\alpha}|=k} 
\mu |\nabla^{\bold{\alpha}} f|^q= 
\sum_{|\bold{\beta}|=k-1} \sum_j
\mu |\nabla_j\nabla^{\bold{\beta}} f|^q 
\geq C_{1,q}^{-1} \sum_{|\bold{\beta}|=k-1}  
\mu | \nabla^{\bold{\beta}} f - M_{1,q}(\nabla^{\bold{\beta}} f)|^q \\
= C_{1,q}^{-1} \sum_{|\bold{\beta}|=k-1}  
\mu | \nabla^{\bold{\beta}} \left(f-B_{1,q,k-1}(f)\right)|^q 
\end{split}
\]
with
\[B_{1,q,k-1}(f) \equiv \sum_{|\bold{\beta}|=k-1} \frac{x^{\bold{\beta}}}{\bold{\beta} !}M_{1,q}(\nabla^{\bold{\beta}} f)\]
By similar arguments applied inductively to the right hand side we obtain
\[\sum_{|\bold{\alpha}|=k} 
\mu |\nabla^{\bold{\alpha}} f|^q \geq C_{1,q}^{-(k-j)}
\sum_{|\bold{\alpha}|=k-j}  
\mu | \nabla^{\bold{\alpha}} \left(f-B_{1,q,j-1}(f)\right)|^q 
\]
with
\[B_{1,q,j-1}(f) \equiv \sum_{|\bold{\alpha}|=k-j} \frac{x^{\bold{\alpha}}}{\bold{\alpha} !}M_{1,q}(\nabla^{\bold{\alpha}} (f-B_{1,q,j}(f))).\]
This implies the desired result with the polynomial
$m_{k,q}\equiv B_{1,q,1}(f)$ and a constant $C_{k,q}\leq C_{1,q}^{k}$.
\end{proof}

The constant in the above result grows exponentially in $k$ and may be not optimal, and different then the optimal constant in the $(PI_{k,q})$ with the minimizing polynomial $M_{k,q}(f)$. At the current stage of knowledge the optimal constants is still a delicate issue. We only know how to determine them for the O-U case.

\bigskip
\label{EndOfPIs}
\newpage

\label{Section 4}
\section{  Orlicz-Sobolev Inequalities}
 \label{Sec.Orlicz-Sobolev Inequalities} 

\noindent\textbf{Shift type minimizers for Orlicz norms.}\\
\label{Shift type minimizers for Orlicz norms}

\noindent Let $\Phi$ be an Orlicz function and let
\[
\|f\|_\Phi \equiv \inf\left\{\lambda > 0: \mu\left(\Phi\left(\frac{f}{\lambda}\right)\right)\leq 1\right\}
\]
be the corresponding Luxemburg norm.


\begin{theorem}  \label{Thm.Minimizers}
Suppose the norm $\|\cdot \|_\Phi$
is strictly convex. Then there exists unique functional  $M_{\Phi,\mu}(f)$ such that 
\[
\|f - M_{\Phi,\mu}(f)\|_\Phi \equiv \inf_{a\in\mathbb{R}} \|f-a\|_\Phi 
\]
Moreover we have   \\
 \[M_{\Phi,\mu}(sf) = |s| M_{\Phi,\mu}(f)\]
and   
\[
\|(sf+tg) - M_{\Phi,\mu}(sf+tg)\|_\Phi \leq
|s| \| f- M_{\Phi,\mu}(f)\|_\Phi 
+ |t| \| g -  M_{\Phi,\mu}(g)\|_\Phi
 \] 
 Let $\Phi$ be an Orlicz function. Then we have
\[
|\mu(f) - M_{\Phi,\mu}(f) |\leq \Phi^{-1}(1) \|f - M_{\Phi,\mu}(f)\|_\Phi \leq  \Phi^{-1}(1) \|f - \mu(f)\|_\Phi 
\]
and
\[
\|f - \mu(f)\|_\Phi \leq
\left( 1 + \Phi^{-1}(1)\cdot \|\mathbb{I}\|_\Phi\right) \|f - M_{\Phi,\mu}(f)\|_\Phi 
\]
We have the following continuity property.
\[ \left| \|f - M_{\Phi,\mu}(f)\|_\Phi -\|g - M_{\Phi,\mu}(g)\|_\Phi \right| \leq  
 \|g - f\|_\Phi .
 \]
\end{theorem} 

\begin{remark}
By definition of $M_{\Phi,\mu}(f)$, we have
\[ 
\|f - M_{\Phi,\mu}(f)\|_\Phi \leq \|f - \mu(f)\|_\Phi
\]
\\
\end{remark}

\begin{proof} 
Using the definition of infimum and Minkowski inequality, we have
\[\begin{split}
\|(sf+tg) - M_{\Phi,\mu}(sf+tg)\|_\Phi \leq
 \|(sf+t g)- \left( sM_{\Phi,\mu}(f) + t  M_{\Phi,\mu}(g)\right)\|_\Phi\\
 \leq 
|s| \| f- M_{\Phi,\mu}(f)\|_\Phi 
+ |t| \| g -  M_{\Phi,\mu}(g)\|_\Phi
\end{split}
 \]

First inequality in the second part follows via Jensen inequality from the convexity of $\Phi$ using the definition of the Luxemburg norm.
The right hand side inequality follows from the definition of the minimizer $M_{\Phi,\mu}(f)$.
The proof of second statement is as follows, using Minkowski inequality together with the inequality proven above, we have
\[\begin{split}
\|f - \mu(f)\|_\Phi \leq |\mu(f) - M_{\Phi,\mu}(f) |\cdot \|\mathbb{I}\|_\Phi + \|f - M_{\Phi,\mu}(f)\|_\Phi \\
\leq
\left( 1 + \Phi^{-1}(1)\cdot \|\mathbb{I}\|_\Phi\right) \|f - M_{\Phi,\mu}(f)\|_\Phi .
\end{split}
\]
 
To show the last statement, note that by definition  
\[ 
\|f - M_{\Phi,\mu}(f)\|_\Phi= \inf_{a\in\mathbb{R}}\|f - a\|_\Phi \leq \|f - g\|_\Phi + \|g - M_{\Phi,\mu}(g)\|_\Phi
 \]
 and similarly
\[  \|g - M_{\Phi,\mu}(g)\|_\Phi \leq
\|g - f\|_\Phi + \|f - M_{\Phi,\mu}(f)\|_\Phi.
\]
 Hence
 \[ \left| \|f - M_{\Phi,\mu}(f)\|_\Phi -\|g - M_{\Phi,\mu}(g)\|_\Phi \right| \leq  
 \|g - f\|_\Phi .
 \]
\end{proof}
 

\noindent\textbf{Perturbation of measures and norms.}\\
\label{Perturbation of measures and norms}

\noindent In this paragraph we will discuss the change of objects introduced before when the measure changes.
Therefore we adopt now a notation which explicitly shows dependence of the norm on the measure.
\\

\begin{theorem} \label{ThmNormsPerturbations}

Consider a probability measure
\[ d\nu = e^{-V}d\mu , \]
defined with $\|V\|_\infty<\infty$. Then we have
\[ e^{\inf(V)} \|f\|_{\Phi,\nu}\leq \|f\|_{\Phi,\mu} \leq e^{\sup(V)} \|f\|_{\Phi,\nu}
\]

\end{theorem} 

\begin{proof}: 
For such a measure $\nu$, we have 
\[
1=\nu\left(\Phi\left(\frac{f}{\|f\|_{\Phi,\nu}}\right)\right)\geq  e^{-\sup(U)} \mu \left(\Phi\left(\frac{f}{\|f\|_{\Phi,\nu}}\right)\right) \geq
\mu \left(\Phi\left(\frac{f}{e^{\sup(U)} \|f\|_{\Phi,\nu}}\right)\right)
\]
Hence, using the definition of Luxemburg norm we get
\[ \|f\|_{\Phi,\mu} \leq e^{\sup(U)} \|f\|_{\Phi,\nu}
\]
Similarly 
\[
1=\mu\left(\Phi\left(\frac{f}{\|f\|_{\Phi,\mu}}\right)\right)\geq  e^{\inf(U)} \nu \left(\Phi\left(\frac{f}{\|f\|_{\Phi,\mu}}\right)\right) 
\]
 and since $\int e^{U}d\nu=1$, we have $\inf U< 0 $, which by convexity of $\Phi$ yields
\[
1\geq   \nu \left(\Phi\left(\frac{f}{e^{-\inf(U)}\|f\|_{\Phi,\mu}}\right)\right). 
\]  
From that we conclude that
\[
e^{-\inf(U)}\|f\|_{\Phi,\mu} \geq  \|f\|_{\Phi,\nu}
\]
\end{proof}

\noindent\textbf{Condition for Minimum Point:} \\
\label{Condition for Minimum Point}

\noindent Suppose $\Phi\in\Delta_2$, i.e. satisfies the doubling condition $\Phi(2x)\leq C\Phi(x)$ with some constant $C\in(0,\infty)$ for all $x\in\mathbb{R}$. 
Then we have 
\[
\mu\,\Phi\left(\frac{f}{\|f\|_{\Phi}}\right)=1
\]
This implies that
\[ \frac{d}{da}\mu\,\Phi\left(\frac{f-a}{\|f-a\|_{\Phi}}\right)=0\]
which implies the following condition for an extremal point
\[
\frac{d}{da} \|f-a\|_{\Phi}=
 \frac{\mu\left(\,\Phi'\left(\frac{f-a}{\|f-a\|_{\Phi}}\right)\right)}{\mu\left(\Phi'\left(\frac{f-a}{\|f-a\|_{\Phi}}\right)\frac{f-a}{\|f-a\|_{\Phi}^2}\right)}
\]
Hence we get the following necessary condition for the minimal point $M_\Phi(f)$.\\

\begin{proposition}  
\[
\mu\left(\,\Phi'\left(\frac{f-M_\Phi(f)}{\|f-M_\Phi(f)\|_{\Phi}}\right)\right) = 0
\]
\end{proposition} 
\bigskip

\noindent\textbf{Example: Special case of $L_p$ norms}:\\
\label{Example: Special case of L_p norms}

\noindent For $p>2$ using convexity
we have
\[ 
\|f - M_{p,\mu}(f)\|_p^2 \leq \left(\mu(f) - M_{p,\mu}(f)\right)^2 + (p-1)\|f - \mu(f)\|_p^2.
\]
For $p\in[1,2)$, we have
\[ 
\|f - M_{p,\mu}(f)\|_p^2 \geq \left(\mu(f) - M_{p,\mu}(f)\right)^2 + (p-1)\|f - \mu(f)\|_p^2
\]
and hence
\[ 
\|f - \mu(f)\|_p \leq  (p-1)^{-\frac12}\|f - M_{p,\mu}(f)\|_p  .
\]
In general, if $\mu$ is a probability measure, we have
\[ 
  \left|\mu(f) - M_{p,\mu}(f)\right|\leq \|f - M_{p,\mu}(f)\|_1\leq  \|f - M_{p,\mu}(f)\|_p \leq \|f - \mu(f)\|_p.
\]
Using this we have
\[
\|f - \mu(f)\|_p \leq \|f - M_{p,\mu}(f)\|_p +  \left|\mu(f) - M_{p,\mu}(f)\right|\leq
2\,\|f - M_{p,\mu}(f)\|_p
\]
Hence, we get the following property.
\\

\begin{proposition} 
For any $p\in[1,\infty)$ and a probability measure $\mu$, we have
\[
\frac12\|f - \mu(f)\|_p \leq 
 \sup_{z\in\mathbb{R}}\|(f+z) - M_{p,\mu}(f+z)\|_p \leq   \|f - \mu(f)\|_p
\]
and  
\[ 
 \sup_{z\in\mathbb{R}}\left| (M_{p,\mu}(f+z)-z)  \right|  \leq | \mu f| + \|f - \mu(f)\|_p 
 \]
\end{proposition} 
Later we will consider an increasing family of closed linear subspaces 
$\mathcal{K}_n\subset \mathcal{K}_{n+1}$ and will be interested in corresponding minimizers. 
\[
\|f - M_{\Phi,n,\mu}(f)\|_\Phi \equiv \inf_{\alpha\in\mathcal{K}_n}
\|f - \alpha \|_\Phi
\] 
Since by our assumption $\mathcal{K}_n$ is a linear space,
for any $h\in \mathcal{K}_n$ we have
\[\begin{split}
\|f + h- M_{\Phi,n,\mu}(f+h)\|_\Phi \equiv \inf_{\alpha\in\mathcal{K}_n}
\|f +h - \alpha \|_\Phi \\
= \inf_{\alpha\in\mathcal{K}_n}
\|f  - (\alpha-h) \|_\Phi = \inf_{\alpha\in\mathcal{K}_n}
\|f  -  \alpha  \|_\Phi 
\end{split}
\]
Hence we get the following result.
\\

\begin{proposition}   
For any $h\in \mathcal{K}_n$, we have
\[
\|f + h- M_{\Phi,n,\mu}(f+h)\|_\Phi = \|f - M_{\Phi,n,\mu}(f)\|_\Phi 
 \]
 and
\[
\|M_{\Phi,n,\mu}(f+h)-h\|_\Phi 
\leq \|f\|_\Phi  + \|f - M_{\Phi,n,\mu}(f)\|_\Phi 
 \] 
\end{proposition} 
\textbf{Remark} : In particular if $M_{\Phi,n,\mu}(f)$ is unique, we have
\[\forall h\in \mathcal{K}_n \quad M_{\Phi,n,\mu}(f+h)-h = M_{\Phi,n,\mu}(f)
\]
\label{Section 5}

\begin{definition} \label{Def.OSineq}
We say that a probability measure $\mu$ satisfies
$(\Phi-\Psi)$ Orlicz-Sobolev Inequality iff $\exists\, C\in(0,\infty)$ such that the following relation holds
\[
\|f-M_{\Phi,\mu}(f)\|_{\Phi,\mu} \leq C
\| |\nabla f| \|_{\Psi,\mu}
\]
for any $f$ for which the right hand side is finite. 
\end{definition} 
Remark: Naturally such inequalities are of more interest if the norm $\| \cdot \|_{\Psi,\mu}$ is weaker than $\|  \cdot\|_{\Phi,\mu}$

\begin{theorem}  (Perturbation Lemma)
\label{Perturbation Lemma}%
 Suppose 
\[ d\nu = e^{-V}d\mu , \]
is a probability measure defined with $\|V\|_\infty<\infty$.
If $\mu$ satisfies
$(\Phi-\Psi)$ Orlicz-Sobolev Inequality with a constant $C_\mu\in(0,\infty)$, then also   
it is satisfied by $\nu$ with a constant not larger than
 \[C_\nu = e^{osc(V)}C_\mu.\]
\end{theorem} 

\begin{proof}  Using definition of the minimiser and comparison of Orlicz norms theorem, we have
\[ \|f-M_{\Phi,\nu}(f)\|_{\Phi,\nu} \leq 
\|f-M_{\Phi,\mu}(f)\|_{\Phi,\nu}\leq
 e^{- \inf(U)} \|f-M_{\Phi,\mu}(f)\|_{\Phi,\mu}
\]
Next using $(\Phi-\Psi)$ Orlicz-Poincare Inequality
for the measure $\mu$ together with comparison of Orlicz norms lemma, we have
\[ 
e^{- \inf(U)} \|f-M_{\Phi,\mu}(f)\|_{\Phi,\mu}
\leq e^{- \inf(U)}C
\| |\nabla f| \|_{\Psi,\mu}
\leq 
e^{\sup (U)- \inf(U)} C \| |\nabla f| \|_{\Psi,\nu}
\]
Combining both relations yields the desired result.
\end{proof} 

\noindent\textbf{Example}: \textbf{Log-Sobolev inequality}
\\
\label{Log-Sobolev inequality}

\noindent A probability measure $\mu$ satisfies Log-Sobolev inequality with respect to a sub-gradient $\nabla\equiv (X_1,..,X_k)$ iff $\exists\; c\in(0,\infty)$
\[ 
Ent_\mu(f^2)\equiv \mu\left(f^2\log\frac{f^2}{\mu(f^2)}\right) \leq 
c \mu |\nabla f|^2
\]
for all functions $f$ for which the right hand side is well defined.\\
It  was observed in \cite{BG} [BG] that, since the right hand side is invariant with respect to the shift by a constant, the left hand side can be replaced by a functional
\[
\mathfrak{L}(f)\equiv \sup_{a\in\mathbb{R}} Ent_\mu((f-a)^2) 
\]
with sup well defined (thanks to Rothaus lemma) and the Log-Sobolev inequality is equivalent to the following relation
\[
\|f - \mu f\|_{N}\leq \tilde c\; \| |\nabla f| \|_{\mathbb{L}_2}
\]
with an Orlicz function
\[ N(x) \equiv x^2\log (1+x^2).\]
According the theory developed above this inequality is equivalent to
\[
\|f - M_{N,\mu} f\|_{N}\leq \hat c\; \| |\nabla f| \|_{\mathbb{L}_2} ,
\]
with a constant $\hat c\in(0,\infty)$ independent of $f$,
and this inequality admits simple bounded perturbation theory of the probability measures.
\\

\noindent\textbf{Orlicz-Sobolev Inequalities of Higher Order.}\\

\label{Orlicz-Sobolev Inequalities of Higher Order}

\noindent Let 
\[
\mathcal{H}_k\equiv \{g\in\mathbb{L}_\Phi(\mu) : \nabla^k g=0\}
\]
For $k\in\mathbb{N}$, define a functional of order $k$
\[
\|f - M_{k,\Phi,\mu}(f)\|_\Phi \equiv \inf_{g\in\mathcal{H}_k} \|f-g\|_{\Phi,\mu} 
\]
Recall for the case $k=1$, we could use $\mathbb{L}_2$ minimizer to dominate the minimum. For $k>1$ the choice of $\mathbb{L}_2$ minimizer may be not possible and depends both on $\Phi$ and $\mu$.
This makes the analysis interesting and more flexible.
In this context one can study the following Orlicz-Poincar{\'e} inequalities of higher order
\[
\|f - M_{k,\Phi,\mu}(f)\|_\Phi \leq c_{k,\Phi} \mu |\nabla^k f|^2.
\]

\label{EndOPineq.1}

\label{Section 6}
\section{  Tight Orlicz-Sobolev inequalities of Higher Order} 
 \label{Sec.Tight Orlicz-Sobolev Inequalities}%

In this section we introduce and study a tight analog of Adams bounds in the form of the following inequalities.

\begin{definition}  
\label{Definition: Orlicz-Sobolev Inequality}%
 
We say that a probability measure $\nu$ satisfies
$(\Phi-\Psi)$ Orlicz-Sobolev Inequality of order $k\in\mathbb{N}$
iff $\exists\, C_k\in(0,\infty)$ such that the following relation holds
\[
\|f-M_{k,\Phi,\mu}(f)\|_{\Phi,\mu} \leq C_k
\| |\nabla^k f| \|_{\Psi,\mu}
\]
for any $f$ for which the right hand side is finite.
\end{definition}

\noindent For $j\in\mathbb{N}$, we define 
j-th composition of exponential function $\exp_{j+1}(t) :=  \exp \left(\exp_{j}(t)\right) $, with $\exp_1(t)\equiv e^t$. Then for $t\geq e_j\equiv\exp_{j}(1)$, we set $\log_{j+1} (t) :=\log\left(\log_{j}t\right) $, with a convention that $\log_1\equiv \log$.  Similarly letting $\log^\ast(t) = max \{ 1, \log(t) \} $, we define its $j$-th composition $\log_j^\ast(t)$. 

\begin{lemma} \label{L1}
For any $j\in\mathbb{N}$ and $p \geq 1$,  we have for any $x\in\mathbb{R}$ the following relation
\begin{equation}
 \frac{1}{C_j} \log_j (\gamma_j + |x|^p) \leq \log_j (\gamma_j + |x| )  
\end{equation}
with $C_1=p$ and $\gamma_1 =1$, whereas for $j\geq 2$, \\
with  $C_j=2$ and  $\gamma_{j+1} \equiv \max \{ e_{j},  p, \gamma_j \}$ .
\end{lemma} 
\begin{proof}
We prove the lemma by induction. When $j=1$, taking $\gamma_1 = 1$, we have the following elementary inequality $(a+b)^p \geq a + b^p$ for $b>0$ and $a \geq 1$.
With $a= \gamma\geq\gamma_1\equiv 1$ and $b = |x|$, we get 
\begin{equation}
p \log (\gamma + |x| ) \geq \log (\gamma  +|x|^p)
\end{equation} 
i.e. we obtain the desired estimate with $C_1 = p$. 
Hence, applying $\log$ to both sides above, for $\gamma\geq \max\{e,p\}$, we get
\begin{equation}
2\log_2 (\gamma + |x| )\geq \log p + \log_2 (\gamma + |x| ) \geq \log_2 (\gamma  +|x|^p)
\end{equation}
Next suppose for $j\geq 2$ we have 
\begin{equation}
2 \log_j (\gamma + |x| ) \geq  \log_j (\gamma + |x|^p)
\end{equation}
for all $\gamma\geq\gamma_j\equiv \max\{e_{j-1},  p, \gamma_{j-1} \}$.  For $\gamma\geq \gamma_{j+1}\equiv max \{ e_{j} ,  p, \gamma_j \}$, taking $\log$ on both sides, we have 
\begin{equation}
2 \log_{j+1}  (\gamma  + |x|) \geq \log 2 + \log_{j+1}  (\gamma  + |x|)  \geq \log_{j+1} (\gamma  + |x|^p).  
\end{equation}
\end{proof}

\begin{lemma} \label{L2}
For all $j\in\mathbb{N}$, and all $|x| \geq e_{j-1} $, 
\begin{equation}
\log_j (e |x| ) \leq 2 \log_j^\ast(|x|)
\end{equation}
\end{lemma}
\begin{proof}
If $j=1$, for all $|x|>0$ we have
\begin{equation}
 \log (e |x|) = 1 + \log |x| \leq 2 \log^\ast (|x|)
\end{equation}
Suppose  for a given $j\in\mathbb{N}$ and all $|x|\geq e_{j-1}$, we have
\begin{equation}
\log_j ( e |x| ) \leq  2  \log_j^\ast(|x|)
\end{equation}
For $|x|\geq e_{j} $, taking $\log$ of both sides of the above, we arrive at
\begin{equation}
\begin{aligned}
\log_{j+1}(e|x|) = \log (\log_j (e|x|)) \leq \log (2 \log_j^\ast (|x|) )  \\
\leq  \log 2 + \log_{j+1}^\ast (|x|) \leq 2 \log_{j+1}^\ast (|x|)
\end{aligned}
\end{equation}
\end{proof}

\begin{lemma} \label{L3}

$\exists D_j\in(0,\infty)$ s.t. 
\begin{equation}
\log_j(\gamma_j+|x|) \leq D_j \log_j^\ast(|x|) 
\end{equation}
for all $x\in \mathbb{R}$. 
\end{lemma}
\begin{proof}
When $|x| \geq \gamma_j  $, by Lemma \ref{L2}, we have
\begin{equation}
\log_j (\gamma_j + |x| ) \leq \log_j (2|x|) \leq 2 \log^\ast_j (|x|)
\end{equation}
For $|x| \leq \gamma_j$, the function 
\begin{equation}
\frac{\log_j (\gamma_j + |x| )}{ \log_j^\ast (|x|)}
\end{equation}
is continuous function in $[-\gamma_j, \gamma_j]$. Hence it can be bounded by some constant $D_j$. 
\end{proof}

\begin{corollary} \label{Cor.1}
Let \[\Phi(t) = t \prod_{j=1}^{n} (\log_j(\gamma_j + |t|))^{p_j} \equiv t \boldsymbol{\Theta}(t),\]  
and 
\[\Phi_{A,p}(t) = |t|^p \prod_{j=1}^{n} (\log_j^\ast (|t|))^{p_j}\equiv |t|^p A(\log^\ast t).\]
 Then there exists $C > 0$ such that 
 \[ \Phi(|x|^p) \leq C \Phi_{A,p}(|x|).\]
 
\end{corollary}

\begin{proof}
Using Lemma \ref{L1}-\ref{L3} above, 
\begin{equation}
\begin{aligned}
\Phi(|x|^p) =&  |x|^p \prod_{j=1}^{n} (\log_j(\gamma_j + |x|^p))^{p_j}  \\
\leq &  |x|^p \prod_{j=1}^{n} (C_j D_j \log_j^\ast (|t|))^{p_j} \\
= & C  |x|^p \prod_{j=1}^{n} (\log_j^\ast (|x|))^{p_j}
\end{aligned}
\end{equation}
where $C= \prod_{j=1}^n C_j D_j$. 
\end{proof}

\begin{lemma} \label{L4}
\[ \| f \|_p^p \leq \Phi^{-1}(1) \|f^p\|_{\Phi}\]
\end{lemma}

\begin{proof}
By Jensen's Inequality, 
\begin{equation}
1 =  \int \Phi\left(\frac{|f|^p}{\|f^p\|_{\Phi}}\right) d\mu  
\geq   \Phi \left(\int \frac{|f|^p}{\|f^p\|_{\Phi}} d\mu \right)  
\end{equation}
Applying monotone function $\Phi^{-1}$ to both sides yields
\[
\|f\|_p^p \leq \Phi^{-1}(1)\|f^p\|_{\Phi} .
\]
\end{proof}

\begin{theorem} \label{Thm_AOineq} 
\label{Thm.9}
Suppose the following Adams Inequality holds 
\begin{equation} \label{AI}
\mu(\Phi_{A,p}(f))\leq \tilde{C}_A\; \mu\left|\nabla^k f\right|^p 
+ \tilde{D}_A\; \Phi_{A,p}(\|f\|_p^p)
\tag{AI}
\end{equation}
with some $\tilde{C}_A,\tilde{D}_A\in(0,\infty)$ independent of $f$.
Then
\begin{equation} \label{AOI}
\|\,\left|f\right|^p\,\|_\Phi\leq C' \mu\left|\nabla^k f\right|^p + D'  \|f\|_p^p .\tag{AOI}
\end{equation}
with some $C',D'\in(0,\infty)$  independent of $f$.
Moreover if the following $(p,k)$-Poincar{\'e} inequality holds
\begin{equation} \label{pk_PI}
\| (f-M_{p,k}(f))\|_p^p\leq c_{p,k}\;  \mu\left|\nabla^k f\right|^p \tag{PI$_{p,k}$}
\end{equation}
with some $c_{p,k}\in(0,\infty)$ for all $f$ for which the right hand side is well defined,
then the following tight $(\Phi,p,k)$-Inequality holds
\begin{equation} \label{OSI}
\|\,\left|f - M_{p,k}(f)\right|^p \,\|_\Phi\leq C\; \mu\left|\nabla^k f\right|^p . \tag{OSI}
\end{equation}
Conversely \ref{OSI} implies $(\Phi,p,k)$-Inequality
with a constant
\[
c_{p,k}=C \Phi^{-1}(1)
\]
\end{theorem}
\begin{proof} 
By Lemma \ref{L1}-\ref{L3} the (\ref{AI}) implies the following inequality
\[
\mu\left(\Phi(|f|^p)\right) \leq C' \mu\left|\nabla^k f\right|^p + D' \Phi_{A,p}(\|f\|_p^p).
\]
with some constants $C',D'\in(0,\infty)$ independent of the function $f$ for which the right hand side is well defined. Applying this inequality with a function $f/{\||f|^p\|_\Phi^\frac1p}$ for which the right hands side is equal to one, we obtain
\[\begin{split}
1= \mu\left(\Phi(\left|f/{\||f|^p\|_\Phi^\frac1p}\right|^p)\right)\leq C' \mu\left|\nabla^k f/{\||f|^p\|_\Phi^\frac1p} \right|^p  \\
+ D'\mu\left(\left|f/{\||f|^p\|_\Phi^\frac1p}\right|^p\right)A\left(\mu\left(\left|f/{\||f|^p\|_\Phi^\frac1p}\right|^p\right)\right)\\
\leq 
C' \mu\left|\nabla^k f/{\||f|^p\|_\Phi^\frac1p} \right|^p + D' A(\Phi^{-1}(1))\mu |f|^p / \||f|^p\|_\Phi
\end{split}
\]
and hence using Lemma \ref{L4} we arrive at (\ref{AOI}).
Applying (\ref{AOI}) to a function $f-M_{p,k}(f)$ and using the (\ref{pk_PI}), we conclude with the tight inequality (\ref{OSI}).\\
On the other hand, using (\ref{OSI}) together with Lemma 6 we get (\ref{pk_PI}) with the desired  constant.
 
\end{proof}
\begin{remark}
Recall that in a particular case, see \cite{C} and \cite{W}, if \ref{AOI} holds for $p=2$ with the Dirichlet form of a Pearson generator, then the generator has a discrete spectrum. Thus Poincar{\'e} inequality and, by our down hill induction, 
Poincar{\'e} inequality of higher order hold. 
\end{remark}

\label{Section 7}
\section{  Relation of the norms} \label{NormsRelations}
Let $d\mu = e^ {-U} d\lambda$, where $U$ is a smooth function in $\mathbb{R}^n$. The adjoint of the gradient in the corresponding Hilbert space is given by $\nabla^\ast=-\nabla+\nabla U$. Consider the Friedrichs extension of the following Dirichlet operator 
\[
L\equiv \nabla^\ast \cdot \nabla = - \sum_{j=1,..,n}\nabla_j^2 + \sum_{j=1,..,n} \nabla_j (U) \nabla_j\equiv -\Delta + \nabla U \cdot\nabla\]
  defined initially on the dense set of smooth compactly supported functions. 
Then the extension, denoted later on by the same symbol $L$, is a positive, self-adjoint operator for which $L^{\frac{1}{2}} $ is a well-defined, positive and self-adjoint linear operator.
\\
In this section we study equivalence of norms defined in terms of higher order derivatives and powers of the operator $L$, respectively given as follows
with $k\in\mathbb{N}$ and $p\in[1,\infty)$.
\begin{equation}
\| f \|_{L,k,p} \equiv \| f \|_p + \|  L^{\frac{k}{2}} f \|_p
\end{equation}
and 
\begin{equation}
\| f \|_{k,p}^{\sim} \equiv \| f \|_p + \| \nabla^k f\|_p
\end{equation}
which is equivalent under our assumptions to $\| f \|_{k,p}$ via Theorem \ref{Thm_W_kpNorms}.
We begin from discussing Hilbert space norms corresponding to $p=2$. 
Our first task is to show that using the Adams regularity condition on the function $U$, we can show the desired equivalence for $k\leq 3$. This will also serve as an useful introduction of ideas and techniques which will be pushed up later when discussing the general case. In the general case we will require additional regularity assumption on the log of the density function.
\label{Thm10}
\begin{theorem} \label{Thm.NormsRelat.1} 
Let $d\mu=  e^{-U} d \lambda$, with a Adams regular function $U$, for which $|\nabla U|\to_{|x|\to\infty}\infty$.
Then the  norms $\| f \|_{L,k,2}$, $\| f \|_{k,2}$ and $\| f \|_{k,2}^{\sim}$ defined on $W_{k,2}$, $k\leq 3$  are equivalent.
 
\end{theorem}

\begin{proof}
Because of Theorem \ref{Thm_W_kpNorms},
it is sufficient to show equivalence of $\| f \|_{L,k,2}$ and $\| f \|_{k,2}^{\sim}$.
For $k=1$ we have equality of norms, because 
\[\|  L^{\frac{1}{2}} f \|_2^2 = \mu\left( f L f\right)= \|\nabla f\|_{2}.\]
For $k=2$ we have
\[ \begin{split}
\|  L f \|_2^2 &= \mu\left( f L^2 f\right)= \sum_{i,j}\mu\left( \nabla_i f   (\nabla_i  \nabla_j^\ast) \nabla_j f\right)\\
&= \sum_{i,j} \mu\left( \nabla_i f   \left(-\nabla_i  \nabla_j + \nabla_i  \nabla_j (U)\right) \nabla_j f\right)  \\
&=\sum_{i,j} \mu\left( \nabla_i f   \left(-\nabla_j +\nabla_j(U) \nabla_i  -\nabla_j(U) \nabla_i + \nabla_i  \nabla_j (U)\right) \nabla_j f\right) \\
&=\sum_{i,j} \mu\left( |\nabla_j \nabla_i f |^2 \right) \\
&- \sum_{i,j} \mu\left( \nabla_j(U) (\nabla_i f) \nabla_i \nabla_j(f) \right)\\
 &+ \sum_{i,j} \mu\left( \left(\nabla_i  \nabla_j (U)\right) \nabla_i (f) \nabla_j(f)\right)
\end{split}
\]
By Cauchy-Schwartz inequality and application of Lemma \ref{AdamsLem.2} and Theorem \ref{Thm_W_kpNorms}, we get
\[ \begin{split}
\left| \sum_{i,j} \mu\left( \nabla_j(U) (\nabla_i f) \nabla_i \nabla_j(f) \right)\right| 
\leq \frac12  \sum_{i,j} \mu\left|\nabla_i \nabla_j(f) \right|^2 + \frac12 \mu\left(|\nabla f|^2(1+|\nabla U|)^2\right) \\
\leq \tilde C_1 \|f\|_{2,2}^2
\end{split}
\]
with some constant $\tilde C_1\in(0,\infty)$ independent of $f$.
By similar arguments based on Cauchy-Schwartz inequality and application of Lemma \ref{AdamsLem.2} and Theorem \ref{Thm_W_kpNorms}, with some constant $\tilde C_2\in(0,\infty)$, we get 
\[
 \begin{split}
\left|\sum_{i,j} \mu\left( \left(\nabla_i  \nabla_j (U)\right) \nabla_i (f) \nabla_j(f)\right)\right|\leq 
\mu\left( |\nabla f|^2 \cdot (1+ |\nabla U|)^{2-\varepsilon }  \right)\leq 
\tilde C_2 \|f\|_{2,2}^2
\end{split}
 \] 
 Hence
\[ \| f \|_{L,1,2}^2 \leq C_2 \|f\|_{2,2}^2.
\]
with some constant $0< C_2\leq 2(C+1)$.
For $k=3/2$, we have
\[ 
\begin{split}
\|  L^{\frac{3}{2}} f \|_2^2 = \mu(|\nabla(Lf)|^2)=\mu(| L\nabla (f)+ [\nabla,L](f)|^2)\\
= \sum_i\mu(| L\nabla_i (f)+ [\nabla_i,L](f)|^2)\\
\leq 2\sum_i\mu(| L\nabla_i (f)|^2) + 2\sum_i\mu(|[\nabla_i,L](f)|^2)
\end{split}
\]
where in the last step we have used the operator convexity of the quadratic function.\
The first on the right hand side can be bounded using the case $k=1$.
For the second term, with some constants $C',C\in(0,\infty)$, we have
\[  \begin{split}
\sum_i\mu(|[\nabla_i,L](f)|^2 = \sum_i\mu(|\sum_j \nabla_i\nabla_j(U)\cdot\nabla_j(f)|^2\\
\leq \mu\left(|\nabla(f)|^2(1+|\nabla u|)^{2-\varepsilon}\right)\leq C'\|f\|_{2,2}^2\leq 
C\|f\|_{3,2}^2
\end{split}
\]
Combining all the above we arrive at
\[ \|  f \|_{L,3,2}^2 \leq C_3\|f\|_{3,2}^2
\]
with some constant $C_3\in(0,\infty)$ independent of $f$.\\
 
Conversely, we have
\[\begin{split}
\mu\left(|\nabla^2 f|^2\right)\equiv \sum_{i,j} \mu\left(|\nabla_i\nabla_j f|^2\right)
=
\sum_{i,j} \mu\left(  \nabla_i f\cdot \nabla_j^\ast \nabla_i \nabla_j f\right)\\
=
\sum_{i,j} \mu\left(  \nabla_i^\ast\nabla_i (f)\cdot \nabla_j^\ast  \nabla_j f\right)+
\sum_{i,j} \mu\left(  \nabla_i (f)\cdot [\nabla_j^\ast, \nabla_i] \nabla_j f\right)\\
= \|f\|_{L,1,2}^2 - \sum_{i,j} \mu\left(  \nabla_i (f)\cdot \nabla_j \nabla_i(U)\cdot \nabla_j (f)\right)\\
\leq 
\|f\|_{L,1,2}^2  +  C\mu\left(  |\nabla(f)|^2 (1 + |\nabla U|)^{2-\varepsilon}  \right)
\\
\leq 
\|f\|_{L,1,2}^2  +  C\left( \mu\left(  |\nabla(f)|^2\right) \right)^\varepsilon   \left( \mu\left(  |\nabla(f)|^2 (1 + |\nabla U|)^2  \right)\right)^{\frac{2-\varepsilon}{2}}  \\
\leq \|f\|_{L,1,2}^2  +  \varepsilon C^{1/\varepsilon}K^{\frac{2-\varepsilon}{2}} \mu\left(  |\nabla(f)|^2 \right)+ \frac{2-\varepsilon}{2} \|f\|_{2,2}^2 
\end{split}
\]
Hence 
\[ \begin{split}
\mu\left(|\nabla^2 f|^2\right) \leq \frac{2}{\varepsilon}\|f\|_{L,1,2}^2  + 2 C^{1/\varepsilon}K^{\frac{2-\varepsilon}{2}} \mu\left(  |\nabla(f)|^2 \right)+  \frac{2-\varepsilon}{\varepsilon} \|f\|_2^2 
\\
\leq \left(\frac{2}{\varepsilon}  +   C^{1/\varepsilon}K^{\frac{2-\varepsilon}{2}}\right) \|f\|_{L,1,2}^2+  \frac{2-\varepsilon}{\varepsilon} \|f\|_2^2 
\\
\leq \left(  C^{1/\varepsilon}K^{\frac{2-\varepsilon}{2}} + \frac{4-\varepsilon}{\varepsilon}\right) \|f\|_{L,1,2}^2
\end{split}
\]
For the higher norm 
\[\begin{split}
\mu\left(|\nabla^3 f|^2\right)\equiv \sum_{i,j,k} \mu\left(|\nabla_i\nabla_j \nabla_k f|^2\right)
=
\sum_{i,j,k} \mu\left(  \nabla_i \nabla_j f\cdot \nabla_k^\ast \nabla_k \nabla_i \nabla_j f\right)\\
=
\sum_{i,j,k} \mu\left(  \nabla_i^\ast\nabla_i \nabla_j (f)\cdot \nabla_k^\ast \nabla_k \nabla_j f\right)+
\sum_{i,j,k} \mu\left(  \nabla_i\nabla_j (f)\cdot [\nabla_k^\ast, \nabla_i] \nabla_j\nabla_k f\right)\\
\leq
  \mu\left(  L \nabla  (f)\cdot L \nabla  (f)\right)+
+  
\mu\left(  \sum_{i,k}  |\nabla_k\nabla_i(U)|  |\nabla^2 f|^2\right)
\\
\leq
  \mu\left(  L \nabla  (f)\cdot L \nabla  (f)\right)
+  
C \mu\left(|\nabla^2 f|^2\left(1+|\nabla U|\right)^{2-\varepsilon} \right)
\\
\leq
  \mu\left(  L \nabla  (f)\cdot L \nabla  (f) \right) 
+  
  C\left( \mu\left(  |\nabla^2(f)|^2\right) \right)^\varepsilon   \left( \mu\left(  |\nabla^2(f)|^2 (1 + |\nabla U|)^2  \right)\right)^{\frac{2-\varepsilon}{2}}  \\
\leq  \mu\left(  L \nabla  (f)\cdot L \nabla  f\right) +
  \varepsilon C^{1/\varepsilon} K^{\frac{2-\varepsilon}{2}}\mu\left(  |\nabla^2(f)|^2 \right)+ \frac{2-\varepsilon}{2} \|f\|_{3,2}^2 
\end{split}
\] 
Hence we obtain 
\[\begin{split}
\mu\left(|\nabla^3 f|^2\right)  \leq  \frac2{\varepsilon} \mu \left(  L \nabla  (f)\cdot L \nabla  f\right) +
  2 C^{1/\varepsilon} K^{\frac{2-\varepsilon}{2}}\mu\left(  |\nabla^2(f)|^2 \right)+ \frac{2-\varepsilon}{\varepsilon} \|f\|_2^2 
\end{split}
\] 
For the first term on the right hand side, we have
\begin{equation} \label{est_32}
\begin{split}
\mu \left(  L \nabla  (f)\cdot L \nabla  f\right) =
\mu\left(  f L^3 (f)\right) +
2 \sum_{j} \mu\left(  [L, \nabla_j ](f)\cdot  \nabla_j L  f\right)
\\
\qquad +  \sum_{j} \mu\left(  [L, \nabla_j ](f)\cdot   [L, \nabla_j ](f) \right)
\end{split}
\end{equation} 
The middle one in (\ref{est_32}) can be treated as follows
\[ \begin{split}
\sum_{j} \mu\left(  [L, \nabla_j ](f)\cdot  \nabla_j L  f\right)= 
-\sum_{j,k} \mu\left(   \nabla_j\nabla_k(U) \nabla_k(f)\cdot  \nabla_j L  f\right) \\
\leq  \mu\left(  \sum_{j,k} |\nabla_j\nabla_k(U)|\cdot |\nabla f|\cdot  |\nabla  L  f|\right)\\
\leq  \mu\left(  \left(1+|\nabla U |\right)^{2-\varepsilon}\cdot |\nabla f|\cdot  |\nabla  L  f|\right)\\
\frac12\mu(fL^3 f)+
C^2\mu\left( |\nabla f|^2 \left(1+|\nabla U |\right)^{(2-\varepsilon)2} \right)
\\
\leq \frac12\mu(fL^3 f)+
  C^2\left(\mu\left( |\nabla f|^2\right)\right)^\varepsilon \left(\mu\left( |\nabla f|^2 \left(1+|\nabla U |\right)^{2\cdot 2} \right)\right)^{\frac{2-\varepsilon}2} \\
  \leq 
 \frac12\mu(fL^3 f)+  C^2\left(\mu\left( |\nabla f|^2\right)\right)^\varepsilon 
 \left( D \|f\|_{3,2}^2 \right)^{\frac{2-\varepsilon}2}\\
 \leq \frac12\mu(fL^3 f)+  \varepsilon C^{2/\varepsilon}(2D)^{\frac{2-\varepsilon}2}
 \left(\mu |\nabla f|^2\right) +
\frac12 {\frac{2-\varepsilon}2} \|f\|_{3,2}^2  
\end{split}
\]
with some $C,D\in(0,\infty)$ independent of $f$.
 Finally the last one in (\ref{est_32}) can be treated as follows
 \[ \begin{split}
\sum_{j} \mu\left(  [L, \nabla_j ](f)\cdot   [L, \nabla_j ](f) \right) =
\sum_{j} \mu\left(   \sum_{k,l}\nabla_k\nabla_j (U) \nabla_l\nabla_j (U) \cdot   \nabla_k f     \nabla_l f   \right) \\
\leq \mu\left(   \left(1+|\nabla U|\right)^{(2-\varepsilon)2}       |\nabla f |^2  \right)\\
\leq \left( \mu   |\nabla f |^2  \right)^\varepsilon \left(\mu\left(   \left(1+|\nabla U|\right)^{4}       |\nabla f |^2  \right)\right)^{\frac{2-\varepsilon}{2}}\\
\leq \left( \mu   |\nabla f |^2  \right)^\varepsilon \left(D \|f\|_{3,2}^2\right)^{\frac{2-\varepsilon}{2}} \\
\varepsilon( D)^{\frac{2-\varepsilon}{2}}   \mu   |\nabla f |^2 + {\frac{2-\varepsilon}{2}} \|f\|_{3,2}^2
\end{split}
 \]
 with some constant $D\in(0,\infty)$ independent of $f$.
 Combining the above inequalities we arrive at the desired bound
 \[
 \mu\left(|\nabla^3 f|^2\right) \leq D' \|f\|_{L,3,2}^2
 \]
with some constant $D'\in(0,\infty)$ independent of $f$.\\
\end{proof}

\bigskip

\noindent To consider the more general case we introduce the following assumption on derivatives of $U$ of  higher order up to an order $3\leq m\in\mathbb{N}$.\\

\label{Assumption A_m}
\noindent \textbf{Assumption $($A$_m)$} : \label{RA2}
\\
There exists a constant $K\in(0,\infty)$ such that
\begin{equation} 
\begin{aligned}
\exists K\in(0,\infty)\qquad  & | \nabla^2 U | \leq& K ( 1+ |\nabla U|)^{2-\epsilon}  \qquad \qquad \qquad \\
 \forall 3\leq k \leq m \,  \exists K_k\in(0,\infty)\qquad  & | \nabla^k U | \leq& K_k ( 1+ |\nabla U|)^{k} \text{ , } \qquad \qquad \qquad   \label{e_condition}
\end{aligned}
\end{equation}

\begin{remark}
The assumption \textbf{$($A$_m$ $)$}
includes a vast family of functions $U$.  In particular it includes semibounded polynomials and any $\Phi_j$, $j\in\mathbb{Z}^+$ classes of \textrm{\cite{A}} for which it holds for any $m\in\mathbb{N}$.\\

\end{remark}
\label{Thm11} %
\begin{theorem} \label{EquivThm2}
Suppose the assumption \textbf{$($A$_{\bar m})$} is satisfied and $|\nabla U|\to_{|x|\to\infty}\infty$. Then for any $k\leq \bar m$ 
the norms 
$\| f \|_{L,k,2}$, $\| f \|_{k,2}$ and  $\| f \|_{k,2}^{\sim}$ are equivalent.

\end{theorem}

\begin{proof}
By Adam's Lemma \ref{AdamsLem.2}, the condition \textbf{$($A$_{\bar m})$} implies
\begin{equation}
\int |\nabla^{k} U |^2 f^2 d\mu \leq K \| f \|_{k,2}^2  \label{e_Adams}
\end{equation}
When $m=1,2$, the proof has been done above, see Theorem \ref{Thm.NormsRelat.1} . Assume now it holds for all $m'$ such that $3 \leq m' \leq m-1$. 
By inductive assumption, Leibnitz rule and \eqref{e_Adams}, we have
\begin{equation} \label{Ethm.e1}
\begin{aligned}
\| L^{\frac{m}{2}}f \|_2 =& \| L^{\frac{m-2}{2}}(Lf) \|_2 \\
\leq& C \left( \| Lf \|_2 + \| \nabla^{m-2} (Lf) \|_2 \right)\\
\leq& C_1 ( \| f \|_2 + \| \nabla^2 f \|_2) + \| \nabla^{m-2} (-\nabla^2 f + \nabla U \nabla f) \|_2 \\
\leq& C_1( \| f \|_2 + \| \nabla^2 f \|_2) + \| \nabla^m f \|_2 + \| \nabla^{m-2}(\nabla U \nabla f) \|_2\\
\end{aligned}
\end{equation}
with some constants $C,C_1\in(0,\infty)$ independent of $f$.
For the last term on the right hand side we have
\begin{equation} \label{Ethm.e2}
\begin{aligned}
\| \nabla^{m-2}(\nabla U \nabla f) \|_2^2 \equiv \sum_{j_1,..,j_{m-2}}\| \nabla^{m-2}_{j_1,..,j_{m-2}}(\sum_l \nabla_l U \nabla_l f) \|_2^2 \\
= \sum_{j_1,..,j_{m-2}}\|  \sum_l \left( [\nabla^{m-2}_{j_1,..,j_{m-2}} ,\nabla_l U] \nabla_l f  +     \nabla_l U  \nabla^{m-2}_{j_1,..,j_{m-2}}\nabla_l f \right)\|_2^2\\
\leq 2 \sum_{j_1,..,j_{m-2}}\|  \sum_l \left( [\nabla^{m-2}_{j_1,..,j_{m-2}} ,\nabla_l U] \right) \nabla_l f   \|_2^2 \\
+2 \sum_{j_1,..,j_{m-2}}\|  \sum_l       \nabla_l U  \nabla^{m-2}_{j_1,..,j_{m-2}}\nabla_l f \|_2^2.
\end{aligned}
\end{equation}
For the last term on the right hand side of (\ref{Ethm.e2}) we use Adams' lemma to get
\begin{equation} \label{Ethm.e3}
2 \sum_{j_1,..,j_{m-2}}\|  \sum_l       \nabla_l U  \nabla^{m-2}_{j_1,..,j_{m-2}}\nabla_l f \|_2^2 \leq C_3 \|f \|_{m,2}^2
\end{equation}
with some constant $C_3\in(0,\infty)$ independent of $f$.
For the first term on the right hand side of (\ref{Ethm.e2})  we use the following formula for the commutator which can be proven inductively.
\begin{lemma} \label{CommutatorLemma}
\begin{equation} \label{Ethm.e4}
 [\nabla^{m-2}_{j_1,..,j_{m-2}} ,F] = \sum_{s=1}^{m-2}  \sum_{\mathbf{j}_s, \mathbf{i}_{m-2-s}} \left(\nabla^{s}_{\mathbf{j}_s}  F \right) \; \nabla^{m-2-s}_{\mathbf{i}_{m-2-s}}
\end{equation}
where
$\mathbf{j}_s\equiv \{j_{r_1},..,j_{r_s}\}$ and $\mathbf{i}_{m-2-s} \{j_{r_{s+1}},..,j_{r_{m-2}}\}$ are partitions of $j_1,..,j_{m-2}$,
of cardinality $s$ and $m-2-s$, respectively.
\end{lemma}

Applying this lemma with $F=\nabla_lU$, we have
\begin{equation} \label{Ethm.e4}
 [\nabla^{m-2}_{j_1,..,j_{m-2}} ,\nabla_l U] = \sum_{s=1}^{m-2}  \sum_{\mathbf{j}_s, \mathbf{i}_{m-2-s}} \left(\nabla^{s}_{\mathbf{j}_s}  \nabla_l U \right) \; \nabla^{m-2-s}_{\mathbf{i}_{m-2-s}}
\end{equation}
with corresponding partitions
$\mathbf{j}_s\equiv \{j_{r_1},..,j_{r_s}\}$ and $\mathbf{i}_{m-2-s} \{j_{r_{s+1}},..,j_{r_{m-2}}\}$ of $j_1,..,j_{m-2}$,
of cardinality $s$ and $m-2-s$, respectively.
We get
\begin{equation} \label{Ethm.e5}
\begin{aligned}
& 2 \sum_{j_1,..,j_{m-2}}\|  \sum_l \left( [\nabla^{m-2}_{j_1,..,j_{m-2}} ,\nabla_l U] \right) \nabla_l f   \|_2^2 \\
&=
 2 \sum_{j_1,..,j_{m-2}} \|  \sum_l \sum_{s=1}^{m-2}  \sum_{\mathbf{j}_s, \mathbf{i}_{m-2-s}} \left(\nabla^{s}_{\mathbf{j}_s}  \nabla_l U \right) \; \nabla^{m-2-s}_{\mathbf{i}_{m-2-s}} \nabla_l f   \|_2^2 \\
& \leq  2 (m-2)\sum_{s=1}^{m-2}\sum_{s=1}^{m-2} \sum_{j_1,..,j_{m-2}} \|  \sum_l   \sum_{\mathbf{j}_s, \mathbf{i}_{m-2-s}} \left(\nabla^{s}_{\mathbf{j}_s}  \nabla_l U \right) \; \nabla^{m-2-s}_{\mathbf{i}_{m-2-s}} \nabla_l f   \|_2^2 \\
& \leq  2 (m-2) \sum_{s=1}^{m-2}\sum_{s=1}^{m-2} \sum_{j_1,..,j_{m-2}}
 \sum_{l,l'}   \sum_{\mathbf{j}_s, \mathbf{i}_{m-2-s}, \mathbf{j}_s', \mathbf{i}_{m-2-s}'}
\\
&
\int  \left( \left(\nabla^{s}_{\mathbf{j}_s}  \nabla_l U \right) \; \nabla^{m-2-s}_{\mathbf{i}_{m-2-s}} \nabla_l f \right)
\left(  \left(\nabla^{s}_{\mathbf{j}_s'}  \nabla_l U \right) \; \nabla^{m-2-s}_{\mathbf{i}_{m-2-s}'} \nabla_{l'} f\right)d\mu 
\end{aligned}
\end{equation}
Hence
\begin{equation} \label{Ethm.e5'}
\begin{aligned}
& 2 \sum_{j_1,..,j_{m-2}}\|  \sum_l \left( [\nabla^{m-2}_{j_1,..,j_{m-2}} ,\nabla_l U] \right) \nabla_l f   \|_2^2 \\
& \leq  2 (m-2) \sum_{s=1}^{m-2}\sum_{s=1}^{m-2} \sum_{j_1,..,j_{m-2}}
 \sum_{l,l'}   \sum_{\mathbf{j}_s, \mathbf{i}_{m-2-s}, \mathbf{j}_s', \mathbf{i}_{m-2-s}'}
\\
&
\left(\int  \left| \left(\nabla^{s}_{\mathbf{j}_s}  \nabla_l U \right) \; \nabla^{m-2-s}_{\mathbf{i}_{m-2-s}} \nabla_l f \right|^2
d\mu\right)^\frac12
\left(\int  \left|  \left(\nabla^{s}_{\mathbf{j}_s'}  \nabla_l U \right) \; \nabla^{m-2-s}_{\mathbf{i}_{m-2-s}'} \nabla_{l'} f\right|^2d\mu\right)^\frac12
\end{aligned}
\end{equation}
Since by our assumption about the derivatives of $U$ in the condition \textbf{$($A$_{\bar m}$ $)$} and Adams' lemma \ref{AdamsLem.2}, we have
\begin{equation} \label{Ethm.e5'}
\begin{aligned}
 \int  \left| \left(\nabla^{s}_{\mathbf{j}_s}  \nabla_l U \right) \; \nabla^{m-2-s}_{\mathbf{i}_{m-2-s}} \nabla_l f \right|^2
d\mu  \leq
K \|\nabla^{m-2-s}_{\mathbf{i}_{m-2-s}} \nabla_l f\|_{s+1}^2
 \\
 \int  \left|  \left(\nabla^{s}_{\mathbf{j}_s'}  \nabla_l U \right) \; \nabla^{m-2-s}_{\mathbf{i}_{m-2-s}'} \nabla_{l'} f\right|^2d\mu \leq
K \|\nabla^{m-2-s}_{\mathbf{i}_{m-2-s}'} \nabla_{l'} f\|_{s+1}^2
\end{aligned}
\end{equation}
we conclude that with some constant $C'\in(0,\infty)$, we have
\begin{equation} \label{Ethm.e6}
 2 \sum_{j_1,..,j_{m-2}}\|  \sum_l \left( [\nabla^{m-2}_{j_1,..,j_{m-2}} ,\nabla_l U] \right) \nabla_l f   \|_2^2 \\
\leq 
C' \|f\|_{m,2}^2 
\end{equation}
Combining (\ref{Ethm.e2})-(\ref{Ethm.e6}) and using Theorem \ref{Thm_W_kpNorms}  we conclude  that
\[
\|L^{\frac{m}{2}}f\|_2 \leq C" \|f\|_{m,2}
\]
with some constant $C"\in(0,\infty)$ independent of $f$.
This ends the proof of the inductive step and implies the first part of the theorem. \label{1stPartQED}
 
\end{proof}

 \label{2ndPartQED}

\noindent In order to proceed further we will need to have a control on constants in the Adam's bounds as follows.

\label{Thm12} 
\begin{theorem}  \label{Revised Adam's Inequality}
Let $m \geq 2$ be an integer.
Suppose the assumption \textbf{$($A$_m)$} holds.
Then $\exists \gamma >0$, $\forall p \in [2- \gamma, \infty)$, $\forall \epsilon>0$, $\exists K>0$, $\forall f \in W_{k,p}(\mu)$, 
\begin{equation}
\int |\nabla^k U |^p |f|^p d\mu \leq \epsilon \| \nabla^k f \|_p^p + K \| f \|_p^p
\end{equation}
for all $k\leq m$.
\end{theorem}
\; %

\noindent
The proof of this theorem is based on the following lemma. 

\begin{lemma} \label{Lemma 8}   
Under the assumption \textbf{$($A$_m)$}, we have \\
$\exists \gamma >0$, $\forall p \in [2- \gamma, \infty)$, $\exists K>0$, $\forall f \in W_{m,p}(\mu)$,
\begin{equation}
\int ( 1+ |\nabla U|)^{2^{m-1} p} | f|^p d\mu \leq K \| f \|_{m,p}^p
\end{equation}
where $\| f \|_{m,p}\equiv \| f \|_{m,p, \mu}$ is the weighted Sobolev norm in the space $W_{m,p}(\mu)$. 
\end{lemma}


\begin{proof}[Proof of Lemma \ref{Lemma 8} ] 
\label{Lemma 8 Proof}
Arguing as at the beginning of the proof of Theorem \ref{AdamsLem.2}  
for $U\equiv U(d)$, we have
\begin{equation}
\int f(1+|\nabla U|) d \mu \leq \int|\nabla f| d \mu
\end{equation}
Replacing $f$ with $|f|^p (1+|\nabla U|)^{2^{m-1}p-1}$, we have
\begin{equation} \label{Lm8.1}
\begin{aligned}
\int |f|^p (1+|\nabla U|)^{2^{m-1}p} d\mu \leq & \int |\nabla (|f|^p (1+|\nabla U|)^{2^{m-1}p-1})d\mu \\
 \leq&  p \int |\nabla f| |f|^{p-1}  (1+|\nabla U|)^ {2^{m-1}p-1} d\mu \\
+&(2^{m-1}p-1) \int |f|^p  (1+|\nabla U|) ^{2^{m-1}p-2} |\nabla^2 U| d\mu
\end{aligned}
\end{equation}
For the first term on the right hand side of \eqref{Lm8.1}, under the assumption $p \geq 2-2^{2-m}$, $\forall \epsilon >0$, $\exists C(\epsilon) >0$ such that
\begin{equation}
\begin{split}
& \int |\nabla f| |f|^{p-1}  (1+|\nabla U|)^ {2^{m-1}p-1} d\mu \\
 & \leq  \int |f|^{p-1}  (1+|\nabla U|) ^{2^{m-1}(p-1)} \cdot |\nabla f |  (1+|\nabla U|) ^{2^{m-2}p} d\mu \\
& \quad \leq  \epsilon \int |f|^p  (1+|\nabla U|) ^{2^{m-1}p} d\mu  
+  C(\epsilon) \int |\nabla f|^p  (1+|\nabla U|) ^{2^{m-2} p} d\mu
\end{split}
\end{equation}
For the second term on the right hand side of \eqref{Lm8.1}
\begin{equation}
\begin{aligned}
\int |f|^p  (1+|\nabla U|) ^{2^{m-1}p-2} |\nabla^2 U| d\mu \leq & C \int |f|^p  (1+|\nabla U|) ^{2^{m-1}p - \delta} d\mu \\
\leq & \epsilon  \int |f|^p  (1+|\nabla U|) ^{2^{m-1}p} d\mu  
+  C(\epsilon) \int |f|^p d\mu 
\end{aligned}
\end{equation}
Taking $\epsilon = \frac{1}{4}$ in both estimates above, we arrive at
\begin{equation}
\frac{1}{2} \int |f|^p  (1+|\nabla U|) ^{2^{m-1}p} d\mu \leq K\left(  \int |\nabla f|^p  (1+|\nabla U|) ^{2^{m-2}p} d\mu + \int |f|^p d\mu \right)
\end{equation}
From here we proceed by induction to conclude with the following bound
\begin{equation}
\int ( 1+ |\nabla U|)^{2^{m-1} p} | f|^p d\mu \leq K \| f \|_{m,p, \mu}^p
\end{equation}

In fact, if $p \geq 2- 2^{2-m}$, then for all $j \leq m$, $p \geq 2-2^{2-j}$, hence
\begin{equation}
\begin{aligned}
\int |f|^p  (1+|\nabla U|) ^{2^{m-1}p} d\mu \leq&  K\left(  \int |\nabla f|^p  (1+|\nabla U|) ^{2^{m-2}p} d\mu + \int |f|^p d\mu \right) \\
\leq & K^2 \left(  \int |\nabla^2 f|^p  (1+|\nabla U|) ^{2^{m-3}p} d\mu + \int |f|^p d\mu  + \int |\nabla f|^p d\mu \right) \\
\leq & \cdots \\
\leq & K^{m-2} \left(  \int |\nabla^{m-2} f|^p  (1+|\nabla U|) ^{2p} d\mu \right.  \\
+& \left. \int |f|^p d\mu  + \int |\nabla f|^p d\mu+ \cdots + \int |\nabla^{m-3} f|^p d\mu  \right) \\
\leq & K^{m-1} \left(  \int |\nabla^{m-1} f|^p  (1+|\nabla U|) ^{p} d\mu \right.  \\
+& \left. \int |f|^p d\mu  + \int |\nabla f|^p d\mu+ \cdots + \int |\nabla^{m-2} f|^p d\mu  \right) \\
\leq & C  \| f \|_{m,p, \mu}^p
\end{aligned}
\end{equation}
where the last line is due to the Adam's inequality of Lemma \ref{AdamsLem.2}. 
\end{proof}

\begin{remark} \label{REM 2}
From the proof we can see that the inequality holds when $p \in [ 2-2^{2-m} , \infty)$. 
\end{remark}
\noindent Assuming Lemma 8
we complete the proof  of the Theorem 12 as follows.
\begin{proof}[Completion of the Proof  of Theorem \ref{Revised Adam's Inequality}]
\label{Thm 12 Proof}
For any $\epsilon >0$, we have
\begin{equation}
\begin{aligned}
\int |\nabla^m U |^p |f|^p d\mu \leq& C\int  (1+|\nabla U|) ^{(2^{m-1} - \delta_m)p} |f|^p d\mu \\
\leq & \epsilon \int  (1+|\nabla U|) ^{2^{m-1}p} |f|^p d\mu + K(\epsilon) \int |f|^p d\mu \\
\leq  & \epsilon ( \| \nabla^m f \|_p^p + \| f \|_p^p ) + K(\epsilon) \| f \|_p^p \\
\leq &  \epsilon \| \nabla^m f \|_p^p + K \| f \|_p^p
\end{aligned}
\end{equation}

with some constants $K(\epsilon), K\in(0,\infty)$ independent of $f$. (Note that in the first inequality we have actually used a weaker condition than $(A_m)$.)
\end{proof}
\qed

\begin{theorem} \label{Thm 13}
 Let $d \mu=e^{-U} d \lambda,$ with a function $U$ for which $|\nabla U| _{\rightarrow|x| \rightarrow \infty} \infty$ and which satisfies assumption \textbf{$($A$_m)$}, for $m\in \mathbb{N}$. 
 Then the norms $\|f\|_{L, k, 2} $, $\|f\|_{k, 2}$ and   $\|f\|_{k, 2}^{\sim}$ , 
 are equivalent on $W_{k, 2}$, for $k\in \mathbb{N},\; k\leq m $.
\end{theorem}

\begin{proof}
The proof is via induction.
The cases $k=1,2,$ are included into Theorem \ref{Thm.NormsRelat.1}. Suppose the statement holds for all $k$ such that $3 \leq k  \leq m^{\prime}-1 $.  We will show that it also holds for  $m^{\prime}\leq m$. Using the inductive assumption,  with some constants $\bar C, C\in(0,\infty)$, we have
\[ 
\begin{aligned}
\left\|L^{\frac{m}{2}} f\right\|_{2} &=\left\|L^{\frac{m-2}{2}}(L f)\right\|_{2} \\
& \leq \bar C \left( \|L f\|_{2}+\left\|\nabla^{m-2}(L f)\right\|_{2} \right) \\
& \leq C\left(\|f\|_{2}+\left\|\nabla^{2} f\right\|_{2}\right)+ \bar C \left\|\nabla^{m-2}\left(-\nabla^{2} f+\nabla U \nabla f\right)\right\|_{2} \\
& \leq C\left(\|f\|_{2}+\left\|\nabla^{2} f\right\|_{2}\right)+ \bar C \left\|\nabla^{m} f\right\|_{2}+\bar C \left\|\nabla^{m-2}(\nabla U \nabla f)\right\|_{2}
\end{aligned}
\]
For the last term on the right hand side of this inequality, we have
\[
\begin{aligned}
\left\|\nabla^{m-2}(\nabla U \nabla f)\right\|_{2}^{2} \equiv & \sum_{j_{1}, \ldots, j_{m-2}}\left\|\nabla_{j_{1}, \ldots, j_{m-2}}^{m-2}\left(\sum_{l} \nabla_{l} U \nabla_{l} f\right)\right\|_{2}^{2} \\
=& \sum_{j_{1}, \ldots, j_{m-2}}\left\|\sum_{l}\left(\left[\nabla_{j_{1}, \ldots, j_{m-2}}^{m-2}, \nabla_{l} U\right] \nabla_{l} f+\nabla_{l} U \nabla_{j_{1}, \ldots, j_{m-2}}^{m-2} \nabla_{l} f\right)\right\|_{2}^{2} \\
\leq&  2 \sum_{j_{1}, \ldots, j_{m-2}}\left\|\sum_{l}\left(\left[\nabla_{j_{1}, \ldots, j_{m-2}}^{m-2}, \nabla_{l} U\right]\right) \nabla_{l} f\right\|_{2}^{2} \\
+& 2 \sum_{j_{1}, \ldots, j_{m-2}}\left\|\sum_{l} \nabla_{l} U \nabla_{j_{1}, \ldots, j_{m-2}}^{m-2} \nabla_{l} f\right\|_{2}^{2}
\end{aligned}
\]

For the last term on the right hand side, we use the Adams' inequality to get
\[
2 \sum_{j_{1}, . ., j_{m-2}}\left\|\sum_{l} \nabla_{l} U \nabla_{j_{1}, \ldots, j_{m-2}}^{m-2} \nabla_{l} f\right\|_{2}^{2} \leq \tilde C\|f\|_{m, 2}^{2}
\]
with some constant $\tilde C \in(0, \infty)$ independent of $f$. For the first term on the right hand side involving the commutator, we use the formula for the commutator of Lemma \ref{CommutatorLemma}
\[
\left[\nabla_{j_{1}, \ldots, j_{m-2}}^{m-2}, \nabla_{l} U\right]=\sum_{s=1}^{m-2} \sum_{\mathbf{j}_{s}, \mathbf{i}_{m-2-s}}\left(\nabla_{\mathbf{j}_{s}}^{s} \nabla_{l} U\right) \nabla_{\mathbf{i}_{m-2-s}}^{m-2-s}
\]
with $\mathbf{j}_{s} \equiv\left\{j_{r_{1}}, \ldots, j_{r_{s}}\right\}$ and $\mathbf{i}_{m-2-s}\left\{j_{r_{s+1}}, \ldots, j_{r_{m-2}}\right\}$ are partitions of $j_{1}, \ldots, j_{m-2}$
of cardinality $s$ and $m-2-s,$ respectively, to get
\[
\begin{aligned}
& 2 \sum_{j_{1}, \ldots, j_{m-2}}\left\|\sum_{l}\left(\left[\nabla_{j_{1}, \ldots, j_{m-2}}^{m-2}, \nabla_{l} U\right]\right) \nabla_{l} f\right\|_{2}^{2} \\
=& 2 \sum_{j_{1}, \ldots, j_{m-2}}\left\|\sum_{l} \sum_{s=1}^{m-2} \sum_{\mathbf{j}_{s}, \mathbf{i}_{m-2-s}}\left(\nabla_{\mathbf{j}_{s}}^{s} \nabla_{l} U\right) \nabla_{\mathbf{i}_{m-2-s}}^{m-2-s} \nabla_{l} f\right\|_{2}^{2} \\
\leq& 2(m-2) \sum_{s=1}^{m-2} \sum_{s=1}^{m-2} \sum_{j_{1}, \ldots, j_{m-2}}\left\|\sum_{l} \sum_{\mathbf{j}_{s}, \mathbf{i}_{m-2-s}}\left(\nabla_{\mathbf{j}_{s}}^{s} \nabla_{l} U\right) \nabla_{\mathbf{i}_{m-2-s}}^{m-2-s} \nabla_{l} f\right\|_{2}^{2} \\
\leq& 2(m-2) \sum_{s=1}^{m-2} \sum_{s=1}^{m-2} \sum_{j_{1}, \ldots, j_{m-2}} \sum_{l, l^{\prime}} \sum_{\mathbf{j}_{s}, \mathbf{i}_{m-2-s}, \mathbf{j}_{s}^{\prime}, \mathbf{i}_{m-2-s}^{\prime}} \\
\int& \left(\left(\nabla_{\mathbf{j}_{s}}^{s} \nabla_{l} U\right) \nabla_{\mathbf{i}_{m-2-s}}^{m-2-s} \nabla_{l} f\right)\left(\left(\nabla_{\mathbf{j}_{s}^{\prime}}^{s} \nabla_{l} U\right) \nabla_{\mathbf{i}_{m-2-s}^{\prime}}^{m-2-s} \nabla_{l^{\prime}} f\right) d \mu
\end{aligned}
\]

Hence
\[
\begin{aligned}
&2 \sum_{j_{1}, \ldots, j_{m-2}}\left\|\sum_{l}\left(\left[\nabla_{j_{1}, \ldots, j_{m-2}}^{m-2}, \nabla_{l} U\right]\right) \nabla_{l} f\right\|_{2}^{2} \\
\leq& 2(m-2) \sum_{s=1}^{m-2} \sum_{s=1}^{m-2} \sum_{j_{1}, \ldots, j_{m-2}} \sum_{l, l^{\prime}} \sum_{\mathbf{j}_{s}, \mathbf{i}_{m-2-s}, \mathbf{j}_{s}^{\prime}, \mathbf{i}_{m-2-s}^{\prime}} \\
&\left(\int\left|\left(\nabla_{\mathbf{j}_{s}}^{s} \nabla_{l} U\right) \nabla_{\mathbf{i}_{m-2-s}}^{m-2-s} \nabla_{l} f\right|^{2} d \mu\right)^{\frac{1}{2}}\left(\int\left|\left(\nabla_{\mathbf{j}_{s}^{\prime}}^{s} \nabla_{l} U\right) \nabla_{\mathbf{i}_{m-2-s}^{\prime}}^{m-2-s} \nabla_{l^{\prime}} f\right|^{2} d \mu\right)^{\frac{1}{2}}
\end{aligned}
\]
since by our assumption about the derivatives of $U$ and the revised Adams' inequality of Theorem \ref{Revised Adam's Inequality}, we have
\[
\begin{array}{c}
\int\left|\left(\nabla_{\mathbf{j}_{s}}^{s} \nabla_{l} U\right) \nabla_{\mathbf{i}_{m-2-s}}^{m-2-s} \nabla_{l} f\right|^{2} d \mu \leq K\left\|\nabla_{\mathbf{i}_{m-2-s}}^{m-2-s} \nabla_{l} f\right\|_{s+1}^{2} \\
\int\left|\left(\nabla_{\mathbf{j}_{s}^{\prime}}^{s} \nabla_{l} U\right) \nabla_{\mathbf{i}_{m-2-s}^{\prime}}^{m-2-s} \nabla_{l^{\prime}} f\right|^{2} d \mu \leq K\left\|\nabla_{\mathbf{i}_{m-2-s}^{\prime}}^{m-2-s} \nabla_{l^{\prime}} f\right\|_{s+1}^{2}
\end{array}
\]
we conclude that with some constant $C^{\prime} \in(0, \infty)$, we have
\[
2 \sum_{j_{1}, \ldots, j_{m-2}}\left\|\sum_{l}\left(\left[\nabla_{j_{1}, \ldots, j_{m-2}}^{m-2}, \nabla_{l} U\right]\right) \nabla_{l} f\right\|_{2}^{2} \leq C^{\prime}\|f\|_{m, 2}^{2}
\]
This concludes the proof of the fact that $\|\cdot\|_{L,m,2}$ are dominated by
$\|\cdot\|_{m,2}$ and $\|\cdot\|_{m,2}^{\sim}$ norms.\\
\noindent To show converse domination, we need to prove that 
\[
 \| \nabla^m f\|_2 \leq C\left( \| f \|_2 + \| L^{\frac{m}{2}} f\|_2 \right)
\]
with some constant $C\in(0,\infty)$ independent of $f$.
The cases $m\leq 3$ was already proven in 
Theorem  \ref{Thm.NormsRelat.1}. We assume by induction that this holds for all $i$ such that $3 \leq i \leq m-1$. Applying the case $m=2$ to $\nabla^{m-2} f$, we have $\| \nabla^m f \|_2 \leq C \| L \nabla^{m-2} f \|_2$. Using the induction assumption, we have
\begin{equation}
\begin{aligned}
\| L \nabla^{m-2} f \|_2 \leq & \| \nabla^{m-2} Lf \|_2 + \| [L, \nabla^{m-2}] f \|_2\\
\leq & C\left( \| Lf \|_2 + \| L^{\frac{m}{2}} f \|_2 + \| [L, \nabla^{m-2}] f \|_2 \right)
\end{aligned}
\end{equation}
For the last term, by Lemma \ref{CommutatorLemma}, with some constant 
$\hat C$, we have  
\begin{equation}
\| [L, \nabla^{m-2}] f \|_2 \leq \hat C \sum_{j=2}^{m}  \| \nabla^j U \nabla^{m-j} f  \|_2
\end{equation}
For each $2 \leq j \leq m-1$, using  Theorem \ref{Revised Adam's Inequality}, for any $\delta>0$ with some constant $C_\delta\in(0,\infty)$, we have
\begin{equation}
\| \nabla^j U \nabla^{m-j} f \|_2^2 \leq  \delta \| \nabla^{m} f \|_2^2 
+ C_\delta \|  f \|_2^2 ) 
\end{equation}
Combining everything together, we arrive at
\begin{equation}
 \| \nabla^m f\|_2 \leq C \delta \|\nabla^m f \|_2^2 + D(\delta) ( \| f \| _2^2 + \| L^{\frac{m}{2}} f \|_2^2 ) 
\end{equation}
Choosing $\delta = \frac{1}{2C}$, we arrive at the desired result. 
\end{proof}

Next we study the case $p>2$. Before that, let us note that the seminorm 
$\| L^{\frac{j}{2}}f \|_p$ has similar properties as $\| \nabla^j f \|_p$.  
\label{REMARK}
First of all for $p=2$ using spectral theory one can show that, for the non negative selfadjoint operator $L$, any $\delta\in(0,\infty)$ with some constant $C_\delta\in(0,\infty)$, one has
\begin{equation}
\|L^{\frac{k}{2}}f\|_2\leq \delta \|L^{\frac{m}{2}}f\|_2 +C_\delta \|f\|_2
\end{equation}
Since in our context $-L$ is a Markov generator, for any $p\in[1,\infty)$ and any $\lambda\in(0,\infty)$, we have the following property of resolvent
\[
\|(\lambda +L)^kf\|_p\leq  \|f\|_p
\] 
and hence
\[
\| (\lambda +L)^{-k} f \|_p \leq  \|f\|_p
\] 


 \label{YTheorem 5}
\begin{theorem} 
Suppose condition \textbf{$($A$_m)$} holds. 
Then for any $p \in[2, \infty)$, 
\begin{equation}
\left\|\nabla^{m} f\right\|_{p} \leq C( \|f\|_{p}+\left\|L^{\frac{m}{2}} f\right\|_{p})
\end{equation}
\end{theorem}

\begin{proof}
We prove this by induction. For $m=2$,  
\begin{equation}
\left\|\nabla^{2} f\right\|_{p} \leq C( \|f\|_{p}+\left\|L f\right\|_{p})
\end{equation}
and for $m=1$  
\begin{equation}
\left\|\nabla f\right\|_{p} \leq C( \|f\|_{p}+\left\|L^{\frac{1}{2}} f\right\|_{p})
\end{equation}
we refer to Theorem 2.1 in \cite{X}, (see also nice review in \cite{Bo} concerning O-U semigroups and references therein).
Suppose
\begin{equation}
\left\|\nabla^{m} f\right\|_{p} \leq C( \|f\|_{p}+\left\|L^{\frac{m}{2}} f\right\|_{p})
\end{equation}
then by inductive assumption, Minkowski inequality and Theorem \ref{Thm_W_kpNorms}, we get
\begin{equation}
\begin{aligned}
\| \nabla^{m+2} f \|_p \leq& C ( \| L( \nabla^m f) \|_p + \| \nabla^m f \|_p ) \\
\leq & C_1( \| [L, \nabla^m] f \|_p + \| \nabla^m (Lf) \|_p + \|f\|_{p}+\left\|L^{\frac{m}{2}} f\right\|_{p}) \\
\leq & C_2( \| [L, \nabla^m] f \|_p + \|L^{\frac{m+2}{2}} f \|_p + \| Lf \|_p+\|f\|_{p}+\left\|L^{\frac{m}{2}} f\right\|_{p}) \\
\end{aligned}
\end{equation}
with some constants $C_1, C_2\in(0,\infty)$ independent of $f$.
In order to deal with the first term, we use the commutator Lemma \ref{CommutatorLemma} and the revised Adam's inequality of Theorem \ref{Revised Adam's Inequality}, to get for any $\epsilon >0$, 
\begin{equation}
\begin{aligned}
\| [L, \nabla^m] f \|_p \leq & C \sum_{i=1}^{m+1}  \| \nabla^{i+1} U \nabla^{m-i+1} f \|_p \\
\leq & C \sum_{i=1}^{m+1} ( \epsilon \| \nabla^{m+2} f \|_p + K(\epsilon) \| \nabla^{m-i+1} f \|_p) \\
\leq & C (\epsilon \| \nabla^{m+2} f \|_p + \| f \|_p )
\end{aligned}
\end{equation}

Taking $\epsilon< \frac{1}{2C}$ , combining the last estimates, we obtain
\begin{equation}
\left\|\nabla^{m+2} f\right\|_{p} \leq C( \|f\|_{p}+\left\|L^{\frac{m+2}{2}} f\right\|_{p})
\end{equation}
\end{proof}

\label{YTheorem 6}
\begin{theorem}
Suppose condition \textbf{$($A$_m)$} holds. 
Then for any $p \in[2, \infty)$, 
\begin{equation}
 \left\|L^{\frac{m}{2}} f\right\|_{p}\leq C( \|f\|_{p}+\left\|\nabla^{m} f\right\|_{p} )
\end{equation}

\end{theorem}

\begin{proof}

First of all we note that
\[
\begin{aligned}
\|L f\|_{p} \leq& \sum_{j}\left\|\nabla_{j}^{2} f\right\|_{p}+\|(\nabla U) \cdot \nabla f\|_{p} \\
\leq& \sum_{i, j}\left\|\nabla_{i} \nabla_{j} f\right\|_{p}+\sum_{j}\left\|\left(\nabla_{j} U\right) \nabla_{j} f\right\|_{p}
\end{aligned}
\]
thus
\begin{equation}
\|L f\|_{p} \leq C_{1, p}\|f\|_{2, p}
\end{equation}

Next assume that for any $1 \leq m \leq k$ we have the following bound
\[
\left\|L^{m} f\right\|_{p} \leq C_{m, p}\|f\|_{2 m, p}
\]

Then, using the inductive assumption definition of $L$ and the norms $\|\cdot\|_{m, p},$ we have
\[
\begin{aligned}
\left\|L^{m+1} f\right\|_{p}=& \left\|L^{m}(-\Delta+\nabla U \cdot \nabla) f\right\|_{p} \\
\leq& C_{m, p}\left(\|(\Delta) f\|_{2 m, p}+\|(\nabla U \cdot \nabla) f\|_{2 m, p}\right) \\
\leq& C_{m, p}\|f\|_{2(m+1), p}+C_{m, p}\|(\nabla U \cdot \nabla) f\|_{2 m, p}
\end{aligned}
\]
Combining all that we with some constant $C_{m+1, p} \in(0, \infty)$, we have
\[
\left\|L^{m+1} f\right\|_{p} \leq C_{m+1, p}\|f\|_{2(m+1), p}
\]
and hence by principle of mathematical induction for all $m \in \mathbb{N}$.
\end{proof}
\label{ENDsection7}
\section{  Higher Order Decay to Equilibrium}
\label{Higher Order Decay to Equi}
In this section we discuss the decay to equilibrium for higher order gradients. First we show under general assumptions utilised in the previous section we have the exponential decay to equilibrium of higher order gradients.  \\ 

For a semigroup $P_{s} \equiv e^{s L}$ generated by $L=\Delta-\nabla U \cdot \nabla$, we set $f_{t} \equiv P_{t} f$ 

Suppose Poincar{\'e} inequality holds 
\[
m_0 \mu (f-\mu f)^2 \leq \mu |\nabla f|^2 
\]
Then for $n\geq 0$, we have
\[
\frac{d}{dt} \mu(L^n f_t)^2= -2 \mu|\nabla L^n f_t|^2\leq 
-2m_0 \mu(L^n f_t-\mu L^n f_t )^2
\]
and so
\[
\mu(L^n f_t-\mu L^n f_t)^2 \leq e^{-2m_0t} \mu(L^n f-\mu L^n f_t)^2.
\]
Then under the conditions of Theorem 11, there exists a constant $C\in(0,\infty)$ such that for $k\in \mathbb{N}$
\[
\mu(\nabla^k f_t)^2 \leq C \|f_t-\mu f\|_{L,\frac{k}{2},2}^2 \leq
e^{-2m_0t} C \|f_t-\mu f\|_{L,\frac{k}{2},2}^2 
\] 
and hence we conclude with the following result.

\begin{theorem}\label{Thm.1DecEq}
Suppose $U$ is locally bounded, $|\nabla U(x)|\to_{|x|\to\infty} \infty$ and it satisfies assumption $(\textbf{A}_k)$.
\\
Suppose there exists $m_0\in(0,\infty)$ such that
\begin{equation} \label{DecEq_Ineq.2}
m_0 \mu \left( f -\mu f\right)^2\leq 
\mu \left( |\nabla f|^2 \right)
\end{equation}
Then there exists a constant $C'\in(0,\infty)$ such that for any $t\in (0,\infty)$, we have
\begin{equation} \label{DecEq_Ineq.3}
\mu(\nabla^k f_t)^2 \leq C' e^{-2m_0t} 
\left(  \mu(\nabla^k f)^2 + \mu(f-\mu f)^2 \right)
\end{equation}
Since under our conditions there exists a constant $c_{k,2}\in(0,\infty)$
\[
 \mu \left( f - M_{k,2} f\right)^2\leq 
c_{k,2}\mu \left( |\nabla^k f|^2 \right)
\]
so there exists $C''\in(0,\infty)$ such that 
\begin{equation} \label{DecEq_Ineq.4}
\mu(\nabla^k (f-M_{k,2})_t)^2 \leq C'' e^{-2m_0t} 
  \mu(\nabla^k f)^2  
\end{equation} 
\end{theorem}
\hfill$\circ$
\vspace{0.50cm}  


  Motivated by the O-U case where it is possible to obtain explicit control for multiparticle case,
in the rest of this section we discuss some heuristic arguments which suggest that in a class of models a stronger decay can be expected. 
First of all we have the following relation.
\[\begin{split}
\partial_s  P_{t-s}\left|\nabla^{k} f_{s}\right|^{2}=
 P_{t-s}\left(-L \left|\nabla^{k} f_{s}\right|^{2} + 2\nabla^{k} f_{s}\cdot \nabla^{k} L f_{s}\right) \\
=P_{t-s}\left(\left(-L \left|\nabla^{k} f_{s}\right|^{2} + 2\nabla^{k} f_{s}\cdot L \nabla^{k}   f_{s} \right)
 + 2\nabla^{k} f_{s}\cdot [\nabla^{k}, L] f_{s} \right)
\end{split}
\]
Now recall that, since  $L$ is a Markov generator, we have
\[-L \left|\nabla^{k} f_{s}\right|^{2} + 2\nabla^{k} f_{s}\cdot L \nabla^{k}   f_{s} = -2 \left|\nabla^{k+1} f_{s}\right|^{2} 
\]
Next as a corollary of Lemma \ref{CommutatorLemma}, we get
 
 \begin{equation} \label{Ethm.e4a}
 [\nabla^{k}_{j_1,..,j_{k}} ,\nabla_l U] = \sum_{r=1}^{k}  \sum_{\mathbf{j}_r, \mathbf{i}_{k-r}} \left(\nabla^{r}_{\mathbf{j}_r}  \nabla_l U \right) \; \nabla^{k-r}_{\mathbf{i}_{k-r}}
\end{equation}
with corresponding partitions
$\mathbf{j}_r\equiv \{j_{n_1},..,j_{n_r}\}$ and $\mathbf{i}_{k-r} \{j_{n_{r+1}},..,j_{n_{k}}\}$ of $j_1,..,j_{k}$,
of cardinality $r$ and $k-r$, respectively, and hence
 \begin{equation} \label{Ethm.e4a}
 [\nabla^{k}_{j_1,..,j_{k}} ,L] = -\sum_{r=1}^{k}  \sum_{\mathbf{j}_r, \mathbf{i}_{k-r}} \left(\nabla^{r}_{\mathbf{j}_r}  \nabla_l U \right) \; \nabla^{k-r}_{\mathbf{i}_{k-r}}\nabla_l
\end{equation}

From this we obtain

\[\begin{split}
\partial_s  P_{t-s}\left|\nabla^{k} f_{s}\right|^{2}=
 P_{t-s}\left( -2 \left|\nabla^{k+1} f_{s}\right|^{2} 
 - 2k \nabla_j \nabla^{k-1}f_{s}\cdot (\nabla_j\nabla_i U)\nabla_i \nabla^{k-1}f_{s} \right.\\
\left.
- 2\nabla^{k} f_{s}\cdot  \sum_{r=2}^{k}  \sum_{\mathbf{j}_r, \mathbf{i}_{k-r}} \left(\nabla^{r}_{\mathbf{j}_r}  \nabla_l U \right) \; \nabla^{k-r}_{\mathbf{i}_{k-r}}\nabla_l f_{s}\right).
\end{split}
\]
Taking expectations and applying quadratic inequalities, after some rearrangement we get
\begin{equation} \label{DecEq_Ineq.1}
\begin{split}
\partial_s  \mu\left|\nabla^{k} f_{s}\right|^{2} +   2k \mu \left(\nabla_j \nabla^{k-1}f_{s}\cdot (\nabla_j\nabla_i U)\nabla_i \nabla^{k-1}f_{s} \right) -2k\varepsilon \mu\left|\nabla^{k} f_{s}\right|^2 
 \\
 \leq
%
 \frac1{2k\varepsilon}  \sum_{r=2}^{k}  \sum_{\mathbf{j}_r, \mathbf{i}_{k-r}} \mu\left( \left|\nabla^{k-r}_{\mathbf{i}_{k-r}}\nabla_l f_{s}\right|^2\; \left|\nabla^{r}_{\mathbf{j}_r}  \nabla_l U \right|^2 \right).
\end{split}
\end{equation}
Hence we see that first necessary condition for a faster decay is the following: \textit{There exists $m \in(0,\infty)$ such that}
\begin{equation} \label{DecEq_Ineq.2}
m  \mu \left| \nabla^k f\right|^2\leq 
\mu \left(\nabla_j \nabla^{k-1}f_{s}\cdot (\nabla_j\nabla_i U)\nabla_i \nabla^{k-1}f_{s} \right). \tag{C}
\end{equation}
Naturally one can satisfy this condition for strictly convex $U$.
To complete the estimates on the decay to equilibrium one needs to estimate the right hand side of (\ref{DecEq_Ineq.1}) which could be achieved with the techniques developed by us in the earlier sections possibly with additional restriction on $U$. 
This will be studied in more detail elsewhere.


\bigskip

\textit{Data Statement : The research does not include any data.}

\label{TheEND*}
\end{document}